\documentclass[12pt]{article}
\usepackage{graphicx,amssymb}
\usepackage[bookmarks=true]{hyperref}
\newcommand{\dsp}{\displaystyle}
\newcommand{\bd}{\begin{displaymath}}
\newcommand{\be}{\begin{equation}}
\newcommand{\beq}{\begin{eqnarray}}
\newcommand{\ba}{\begin{array}}
\newcommand{\ed}{\end{displaymath}}
\newcommand{\ee}{\end{equation}}
\newcommand{\eeq}{\end{eqnarray}}
\newcommand{\ea}{\end{array}}
\newcommand{\espace}{\mbox{ }}
\newcommand{\norm}[1]{\left|\left|#1\right|\right|}

\newcommand{\eps}{\varepsilon}
\newcommand{\sgn}{\mbox{\rm sgn}\,}

\newcommand{\abs}[1]{\left|#1\right|}
\newcommand{\Prob}{{\rm I\hspace{-0.8mm}P}}
\newcommand{\Exp}{{\rm I\hspace{-0.8mm}E}}
\newcommand{\N}{{\mathbb N}}

\newcommand{\Z}{{\mathbb Z}}
\newcommand{\R}{{\mathbb R}}

\newcommand{\dt}{\partial_t}
\newcommand{\dx}{\partial_x}
\newcommand{\eqref}[1]{(\ref{#1})}

\newtheorem{theorem}{Theorem}[section]

\newtheorem{proposition}{Proposition}[section]

\newtheorem{lemma}{Lemma}[section]
\newtheorem{corollary}{Corollary}[section]

\newtheorem{remark}{Remark}[section]
\newenvironment{proof}[2]{\espace\\{\em Proof of #1 \ref{#2}.}}{\hfill\mbox{$\square$}}
\begin{document}
\title{Strong hydrodynamic limit for  attractive  particle
systems on $\Z$
}
\author{C. Bahadoran$^{a,b}$, H. Guiol$^c$, K. Ravishankar$^d$, E. Saada$^e$}
\date{}
\maketitle
$$ \ba{l}
^a\,\mbox{\small Laboratoire de Math\'ematiques, Univ. Clermont-Ferrand 2, 63177 Aubi\`ere, France} \\
^b\, \mbox{\small Corresponding author. e-mail:
bahadora@math.univ-bpclermont.fr}\\
^c\, \mbox{\small TIMC - TIMB, INP Grenoble, Facult\'{e} de
M\'{e}decine, 38706 La Tronche Cedex, France
} \\
^d\, \mbox{\small Dep. of Mathematics, SUNY,
College at New Paltz, NY, 12561, USA} \\
^e\, \mbox{\small CNRS, UMR 6085, Univ. Rouen, 76801
Saint Etienne du Rouvray Cedex, France}\\
\ea
$$
\begin{abstract}
\noindent  We prove almost sure Euler hydrodynamics for a large
class of attractive particle systems on $\Z$ starting from an
arbitrary initial profile. We generalize earlier works by
Sepp\"al\"ainen (1999) and Andjel et al. (2004). Our constructive
approach requires new ideas since the subadditive ergodic theorem
(central to previous works) is no longer effective in our setting.

\end{abstract}
%
\section{Introduction}
Hydrodynamic limit (\cite{dp, spohn, kl, var0}) is a law of large
numbers for the time evolution (usually described by a limiting PDE,
 called the hydrodynamic equation) of empirical density
fields in interacting particle systems (\emph{IPS}). In most known
results, only a weak law of large numbers is established. In this
description one need not have an explicit construction of the
dynamics: the limit is shown in probability with respect to the law
of the process, which is characterized in an abstract way by its
Markov generator and Hille-Yosida's theorem  (\cite{lig1}).
Nevertheless, when simulating particle systems, one naturally uses a
pathwise construction of the process on a
 Poisson space-time random graph (the so-called
``graphical construction''). In this description the dynamics is
deterministically driven by a random space-time measure which
tells when and where the configuration has a chance of being
modified. It is of special interest to show that the hydrodynamic
limit holds for almost every realization of the space-time
measure, as this means  a single simulation is enough to
approximate solutions of the limiting PDE. In a sense this is
comparable to the interest of proving strong (rather than weak)
consistency for statistical estimators, by which one knows
that a single path observation is enough to estimate parameters.\\ \\
%
Most usual IPS can be divided into two groups, diffusive and
hyperbolic. In the first group,  which contains for instance the
symmetric or mean-zero asymmetric exclusion process, the
macroscopic$\to$microscopic space-time scaling is
$(x,t)\mapsto(Nx,N^2t)$ with $N\to\infty$, and the limiting PDE is a
diffusive equation. In the second group, which contains for instance
the nonzero mean asymmetric exclusion process, the scaling is
$(x,t)\mapsto(Nx,Nt)$, and the limiting PDE is of Euler type. In
both groups this PDE  often  exhibits nonlinearity, either via the
diffusion coefficient in the first group, or via the flux function
in the second one. This raises special difficulties in the
hyperbolic case, due to shocks and non-uniqueness for the solution
of the PDE, in which case the natural problem is to establish
convergence to the so-called ``entropy solution" (\cite{serre}). In
the diffusive class it is not so necessary to specifically establish
strong laws of large numbers, because one has a fairly robust method
(\cite{kov})  to establish large deviations from the hydrodynamic
limit. Large deviation upper bound and Borel Cantelli's lemma imply
a strong law of large numbers as long as the large deviation
functional is lower-semicontinuous and has a single zero. The
situation is quite different in the hyperbolic class, where large
deviation principles are much more difficult to obtain. So far only
the remarkable result of Jensen and Varadhan (\cite{jen} and
\cite{var}) is available, and it applies only to the one-dimensional
   totally asymmetric simple exclusion process. Besides, the
fact that the resulting large deviation functional has a single zero
is not at all obvious: it follows only from recent and refined work
on conservation laws (\cite{dow}).  An even more difficult situation
occurs for particle systems with nonconvex and nonconcave flux
function, for which the Jensen-Varadhan large deviation functional
does not have a unique zero, due to the existence of non-entropic
solutions satisfying a single entropy inequality
(\cite{hlf}).  In this case a more complicated rate functional can actually be conjectured from \cite{bbmn, mar}. \\ \\
%
 The derivation of hyperbolic equations as hydrodynamic limits
began with the seminal paper \cite{rost}, which established a strong
law of large numbers for the totally asymmetric simple exclusion
process on $\Z$, starting with $1$'s to the left of the origin and
$0$'s to the right. This result was extended by \cite{bf} and
\cite{afs} to nonzero mean exclusion process starting from product
Bernoulli distributions with arbitrary densities $\lambda$ to the
left and $\rho$ to the right (the so-called  ``Riemann'' initial
condition). The Bernoulli distribution at time $0$ is   related to
the fact that uniform Bernoulli measures are invariant for the
process. For the one-dimensional totally asymmetric $K$-exclusion
process, which does not have explicit invariant measures,  a strong
 hydrodynamic limit was established in \cite{timo}, starting
from arbitrary initial profiles, by means of the so-called
``variational coupling''.
These are the only strong laws available
so far. A common feature of these works is the use of subadditive
ergodic theorem to exhibit some a.s. limit, which is then identified by additional arguments. \\ \\
%
 On the other hand, many {\em weak} laws of large numbers were
established for attractive particle systems. A first series of
results treated  systems with product invariant measures and product
initial distributions. In \cite{ak}, for a particular zero-range
model, a weak law was deduced from conservation of local equilibrium
under Riemann initial condition. It was then extended in \cite{av}
to the misanthrope's process of \cite{coc} under an additional
convexity assumption on the flux function. These were substantially
generalized (using Kru\v{z}kov's entropy inequalities, see
\cite{kru}) in \cite{fraydoun} to multidimensional attractive
systems with arbitrary Cauchy data, without any convexity
requirement on the flux.
For systems with unknown invariant measures, the result of
\cite{timo} was later extended to other models, though only through
weak laws.
In \cite{rez}, using semigroup point of view, hydrodynamic limit was
established for the one-dimensional nearest-neighbor $K$-exclusion
process.  In \cite{bgrs2} we studied fairly general attractive
systems, employing a constructive approach we had initiated in
\cite{bgrs1}. This method is based on an approximation scheme (to go
from Riemann to Cauchy initial data) and control of some distance
between the particle system and the entropy solution to the
hydrodynamic equation. \\ \\
%
%
%
%
In this paper we prove a strong law of large numbers, starting from
arbitrary initial profiles, for finite-range attractive particle
systems on $\Z$ with bounded occupation number. We need no
assumption on the flux, invariant measures or microscopic structure
of initial distributions. This includes finite-range exclusion and
$K$-exclusion processes, more general misanthrope-type models and
$k$-step exclusion processes.
We proceed in the constructive spirit of \cite{bgrs1, bgrs2}, which
is adapted to a.s. convergence, though this requires a novel
approach for the Riemann part, and  new error analysis for the
Cauchy part.
%
%
 The implementation of an  approximation scheme involves
Riemann hydrodynamics from \emph{any space-time shifted initial
position}, which cannot be obtained through subadditive ergodic
theorem, and thus prevents us from using the a.s. result of
\cite{afs}. This can be explained as follows (see Appendix
\ref{remarks_sub} for more details). Assume that on a probability
space we have a shift operator $\theta$ that preserves the
probability measure, and a sequence $(X_n)_{n\in\N}$ of real-valued
random variables converging a.s. to a constant $x$. Then we can
conclude that the sequence $(Y_n:=X_n\circ\theta^n)_n$ converges
{\em in probability} to $x$, but not necessarily almost surely.
Indeed, for fixed $n$, $Y_n$ has the same distribution as $X_n$, but
the sequence $(Y_n)_n$ need not have the same distribution as
$(X_n)_n$. So the derivation of a.s. convergence for $Y_n$ is a case
by case problem depending on how it was obtained for $X_n$. In
particular, if convergence of $X_n$ was established from the
subadditivity property
$$
X_{n+m}\leq X_n+X_m\circ\theta^n
$$
this property is no longer satisfied by $Y_n$. In contrast, if
convergence of $X_n$ was established by large deviation estimates
for $X_n$,  these estimates carry over to $Y_n$, thus also implying
a.s. convergence of $(Y_n)_n$.
In our context the random variable $X_n$ is a current in the Riemann
setting. In order to handle the Cauchy problem we have to establish
a.s. convergence for a shifted version $Y_n$ of this current. We
derive asymptotics for $Y_n$  by means of large
deviation arguments.\\ \\
%
%
The paper is organized as follows. In Section \ref{sec_results}, we
define the model, give its graphical representation, its
monotonicity properties, and state strong hydrodynamics. In Section
\ref{riemann}, we derive almost sure Riemann hydrodynamics. In
Section \ref{general} we prove the result starting from any initial
profile.

%
%
\section{Notation and results}\label{sec_results}
Throughout this paper $\N=\{1,2,...\}$ will denote the set of
natural numbers, $\Z^+=\{0,1,2,...\}$ the set of non-negative
integers, and $\R^{+*}=\R^+-\{0\}$ the set of positive real numbers.
We consider particle systems on $\Z$ with at most $K$ particles per
site, $K\in\N$. Thus the state space of the process is ${\mathbf
X}=\{0,1,\cdots,K\}^{\Z}$, which we endow with the product topology,
that makes ${\mathbf X}$ a compact metrisable space.   A function
defined on ${\mathbf X}$ is called {\em local} if it depends on the
variable $\eta\in{\mathbf X}$ only through $(\eta(x),x\in\Lambda)$
for some finite subset $\Lambda$ of $\Z$. We denote by $\tau_x$,
either the spatial translation operator on the real line for
$x\in\R$, defined by $\tau_x y=x+y$, or its restriction to $\Z$ for
$x\in\Z$.  By extension, we set $\tau_x f=f\circ\tau_x$ if $f$ is a
function defined on $\R$; $\tau_x\eta=\eta\circ\tau_x$, for
$x\in\Z$, if $\eta\in\bf X$ is a particle configuration (that is a
particular function on $\Z$); $\tau_x\mu=\mu\circ\tau_x^{-1}$ if
$\mu$ is a measure on $\R$ or on $\bf X$.   We let ${\mathcal
M}^+(\R)$ denote  the set of positive measures on $\R$ equipped with
the metrizable topology of vague convergence, defined by convergence
on continuous test functions with compact support. The set of
probability measures on $\mathbf{X}$ is denoted by ${\mathcal
P}(\mathbf{X})$.  If $\eta$ is an ${\mathbf X}$-valued random
variable and $\nu\in{\mathcal P}(\mathbf{X})$, we write
$\eta\sim\nu$ to specify that $\eta$ has distribution $\nu$. The
notation $\nu(f)$, where $f$ is a real-valued function and
 $\nu\in{\mathcal P}(\mathbf{X})$, will be an
alternative for $\int_{\mathbf X}fd\nu$.
We say a sequence $(\nu_n,n\in\N)$ of probability measures on
${\mathbf X}$ converges weakly to some $\nu\in{\mathcal
P}(\mathbf{X})$, if and only if $\nu_n(f)\to\nu(f)$ as $n\to\infty$
for every continuous function  $f$ on ${\mathbf X}$.  The topology
of weak convergence is metrizable and makes ${\mathcal P}({\bf X})$
compact.
\subsection{The system and its graphical construction}
\label{graphical_construction}
We consider Feller processes on $\bf X$ whose transitions consist of
particle jumps encoded by the Markov generator
\be \label{generator}
Lf(\eta)=\sum_{x,y\in{\Z}}p(y-x)b(\eta(x),\eta(y)) \left[
f\left(\eta^{x,y} \right)-f(\eta) \right] \ee
for any local function $f$, where $\eta^{x,y}$ denotes the new state
after a particle has jumped from $x$ to $y$ (that is
$\eta^{x,y}(x)=\eta(x)-1,\,\eta^{x,y}(y)=\eta(y)+1,\,
\eta^{x,y}(z)=\eta(z)$ otherwise), $p$ is the particles' jump
kernel, that is  a probability distribution on $\Z$, and $b\ :\
\Z^+\times\Z^+\to \R^+$ is the jump rate. We assume that $p$ and $b$
satisfy :
\begin{tabbing}
{\em (A1)} \= The semigroup of $\Z$ generated by the support of
$p$ is $\Z$ itself\\
\> (irreducibility);\\
{\em (A2)} \> $p$ has a finite first moment,
that is $\sum_{z\in\Z}\abs{z}p(z)<+\infty$;\\
{\em (A3)} \> $b(0,.)=0,\,b(.,K)=0$ (no more than $K$
particles per site), and\\
\> $b(1,K-1)>0$;\\
{\em (A4)} \> $b$ is nondecreasing (nonincreasing) in its first
(second) argument.
 \end{tabbing}
We denote by $(S(t),t\in\R^+)$ the semigroup generated by $L$.
Without additional algebraic relations satisfied by $b$ (see
\cite{coc}), the system in general has no explicit invariant
measures. This is the case even for {\em $K$-exclusion process}
 with $K\geq 2$ (see \cite{timo}), where $b(n,m)={\bf
1}_{\{n>0,m<K\}}$.
\begin{remark}
For the sake of simplicity, we restrict our attention to processes
of ``misanthrope type'',  for which particles leave more likely
crowded sites for less occupied ones  (cf. \cite{coc}). But our
method can be applied to a much larger class of attractive models
that can be constructed through a graphical representation, such as
$k$-step exclusion (see \cite{guiol} and \cite{bgrs1}).
\end{remark}
We now describe the graphical construction of the  system given by
\eqref{generator}, which uses the Harris representation
(\cite{h,har}, \cite[p. 172]{lig1}, \cite[p. 119]{bd}, \cite[p.
215]{lig2}); see for instance \cite{afs} for details and
justifications. This enables us to define the evolution from
different initial configurations simultaneously on the same
probability space.
 We consider the probability space $(\Omega,{\mathcal
F},\Prob)$ of measures $\omega$ on $\R^+\times\Z^2\times[0,1]$ of
the form
\[
\omega(dt,dx,dz,du)=\sum_{m\in\N}\delta_{(t_m,x_m,z_m,u_m)}
\]
where  $\delta_{(\cdot)}$ denotes Dirac measure, and
$(t_m,x_m,z_m,u_m)_{m\ge 0}$ are pairwise distinct and form a
locally finite set. The $\sigma$-field $\mathcal F$ is generated
by the mappings $\omega\mapsto\omega(S)$ for Borel sets $S$. The
probability measure $\Prob$ on $\Omega$ is the one that makes
$\omega$ a Poisson process with intensity
\[
m(dt,dx,dz,du)=
||b||_\infty\lambda_{\R^+}(dt)\times\lambda_{\Z}(dx)\times
p(dz)\times\lambda_{[0,1]}(du) \]
where $\lambda$ denotes either the Lebesgue or the counting
measure.
We denote by $\Exp$ the corresponding expectation. Thanks to
assumption {\em (A2)}, for $\Prob$-a.e. $\omega$, there exists a
unique mapping
\be \label{unique_mapping} (\eta_0,t)\in{\bf
X}\times\R^+\mapsto\eta_t=\eta_t(\eta_0,\omega)\in{\bf X} \ee
satisfying: (a) $t\mapsto\eta_t(\eta_0,\omega)$ is
right-continuous; (b) $\eta_0(\eta_0,\omega)=\eta_0$; (c) for
$t\in\R^+$, $(x,z)\in\Z^2$, $\eta_t=\eta_{t^-}^{x,x+z}$ if
\begin{equation}
\label{update} \exists u\in[0,1]:\,\omega\{(t,x,z,u)\}=1\mbox{ and
} u\leq\displaystyle{\frac{b(\eta_{t^-}(x),\eta_{t^-}(x+z))}
{||b||_\infty}}
\end{equation}
and (d) for all $s,t\in\R^{+*}$ and $x\in\Z$,
\begin{equation}
\label{noupdate} \omega\{[s,t]\times Z_x\times (0,1)\}=0 \Rightarrow
\forall v\in[s,t],\eta_v(x)=\eta_s(x)
\end{equation}
where
$$
Z_x:=\left\{(y,z)\in\Z^2: y=x \mbox{ or } y+z=x \right\}
$$
In short, \eqref{update} tells how the state of the system can be
modified by an ``$\omega$-event'', and \eqref{noupdate} says that
the system cannot be modified outside $\omega$-events. This defines
a Feller process with generator \eqref{generator}:
 that is for any $t\in\R^+$ and $f\in C(\mathbf{X})$
(the set of continuous functions on $\mathbf{X}$), $S(t)f\in
C(\mathbf{X})$ where
$S(t)f(\eta_0)=\Exp[f(\eta_t(\eta_0,\omega))]$.\\ \\
An equivalent formulation  
is the following. For each $(x,z)\in\Z^2$, let $\{T_n^{x,z},n\geq
1\}$ be the arrival times of mutually independent rate
$||b||_\infty p(z)$ Poisson processes, let $\{U_n^{x,z},n\geq 1\}$
be mutually independent (and independent of the Poisson processes)
random variables, uniform on $[0,1]$. At time $t=T_n^{x,z}$, the
configuration $\eta_{t^-}$ becomes $\eta_{t^-}^{x,x+z}$ if
$\displaystyle{U_n^{x,z}\le
\frac{b(\eta_{t^-}(x),\eta_{t^-}(x+z))}{||b||_\infty}}$, and stays
unchanged otherwise.
\begin{remark} For the $K$-exclusion process,
 $b$ takes its values in $\{0,1\}$,
the probability space can be reduced to measures
$\omega(dt,dx,dz)$ on $\R^+\times\Z^2$, and \eqref{update} to
$\eta_t=\eta_{t^-}^{x,x+z}$ if  $\eta_{t^-}(x)>0$ and
$\eta_{t^-}(x+z)<K$. One recovers exactly the graphical
construction presented in e.g. \cite{timobook} or \cite{afs}.
\end{remark}
One may further introduce an ``initial'' probability space
$(\Omega_0,{\mathcal F}_0,\Prob_0)$, large enough to construct
random initial configurations $\eta_0=\eta_0(\omega_0)$ for
$\omega_0\in\Omega_0$.
The general process with random initial configurations is
constructed on the enlarged space
$(\widetilde{\Omega}=\Omega_0\times\Omega,\widetilde{\mathcal
F}=\sigma({\mathcal F}_0\times{\mathcal
F}),\widetilde{\Prob}=\Prob_0\otimes\Prob)$ by setting
\[
\eta_t(\widetilde{\omega})=\eta_t(\eta_0(\omega_0),\omega)\]
for $\widetilde{\omega}=(\omega_0,\omega)\in\widetilde{\Omega}$.
If $\eta_0$ has distribution $\mu_0$, then the process thus
constructed is Feller with generator \eqref{generator} and initial
distribution $\mu_0$.  By a {\em coupling} of two systems, we mean
a process $(\eta_t,\zeta_t)_{t\geq 0}$ defined on
$\widetilde\Omega$, where each component evolves according to
 \eqref{unique_mapping}--\eqref{noupdate}, and the
random variables $\eta_0$ and $\zeta_0$ are defined
simultaneously on $\Omega_0$.\\ \\
We define on $\Omega$ the {\sl space-time shift} $\theta_{x_0,t_0}$:
for any $\omega\in\Omega$, for any $(t,x,z,u)$

\[
(t,x,z,u)\in\theta_{x_0,t_0}\omega\mbox{ if and only if
}(t_0+t,x_0+x,z,u)\in\omega\]
 where $(t,x,z,u)\in\omega$ means $\omega\{(t,x,z,u)\}=1$. By its very definition, the mapping introduced in
\eqref{unique_mapping} enjoys the following properties, for all
 $s,t\geq 0$, $x\in\Z$ and $(\eta,\omega)\in{\bf
X}\times{\Omega}$:
\be \label{mapping_markov}
\eta_s(\eta_t(\eta,\omega),\theta_{0,t}\omega)=\eta_{t+s}(\eta,\omega)
\ee
which  implies Markov property, and
\be \label{mapping_shift}
\tau_x\eta_t(\eta,\omega)=\eta_t(\tau_x\eta,\theta_{x,0}\omega)
\ee
which implies that $S(t)$ and $\tau_x$ commute.
\subsection{Main result}
 We give here a precise statement of strong
hydrodynamics. Let $N\in\N$ be the scaling parameter for the
hydrodynamic limit, that is the inverse of the macroscopic distance
between two consecutive sites. The empirical measure of a
configuration $\eta$ viewed on scale $N$ is given by
\[
\alpha^N(\eta)(dx)=N^{-1}\sum_{y\in\Z}\eta(y)\delta_{y/N}(dx)
\in{\mathcal M}^+(\R) \]
We now state our main result.
\begin{theorem}\label{th:hydro}
Assume $p(.)$ is finite range, that is there exists $M>0$ such
that $p(x)=0$ for all $|x|>M$. Let $(\eta^N_0,\,N\in\N)$ be a
sequence of $\bf X$-valued random variables on $\Omega_0$. Assume
there exists a measurable $[0,K]$-valued profile $u_0(.)$ on $\R$
such that
\be\label{initial_profile_vague}
\lim_{N\to\infty}\alpha^N(\eta^N_0)(dx)=
u_0(.)dx,\quad\Prob_0\mbox{-a.s.}\ee
that is,
\be \label{initial_profile}\lim_{N\to\infty}
\int_\R\psi(x)\alpha^N(\eta^N_0)(dx) =\int
\psi(x)u_0(x)dx,\quad\Prob_0\mbox{-a.s.} \ee
for every continuous function $\psi$ on $\R$ with compact support.
 Let $(x,t)\mapsto u(x,t)$ denote the unique entropy
solution to the scalar conservation law
\be \label{hydrodynamics} \dt u+\dx[G(u)]=0 \ee
with initial condition $u_0$, where $G$ is a Lipschitz-continuous
flux function  (defined in \eqref{flux} below) determined by $p(.)$
and $b(.,.)$.
Then,  with $\widetilde\Prob$-probability one, the convergence
\be \label{later_profile}
\lim_{N\to\infty}\alpha^N(\eta_{Nt}(\eta^N_0(\omega_0),\omega))(dx)=u(.,t)dx
\ee
holds uniformly on all bounded time intervals. That is, for every
continuous function $\psi$ on $\R$ with compact support, the
convergence
\be \label{later_profile_test}\lim_{N\to\infty}
\int_\R\psi(x)\alpha^N(\eta^N_{Nt})(dx) =\int \psi(x)u(x,t)dx\ee
 holds uniformly on all bounded time intervals.
\end{theorem}
 As a byproduct we derive  
 in Appendix \ref{appendix_strong-weak} the main result of \cite{bgrs2}:
\begin{corollary}\label{strong-weak} The strong law of large numbers \eqref{later_profile}
implies the weak law of large numbers established in \cite{bgrs2}.
\end{corollary}
We recall from \cite[pp.1346--1347 and Lemma 4.1]{bgrs2} the
definition of the Lipschitz-continuous macroscopic flux function
$G$. For  $\rho\in\mathcal R\subseteq [0,K]$  (the  closed
 set of allowed densities for the system, defined in
Proposition \ref{invariant} below)  let \be \label{flux}
G(\rho)=\nu^\rho\left[\sum_{z\in\Z}zp(z)b(\eta(0),\eta(z))
\right]; \ee
 this represents the expectation, under the shift invariant
equilibrium measure with density $\rho$ (see Proposition
\ref{invariant}), of the microscopic current through site
$0$. 
On the complement of $\mathcal R$, which is at most a countable
union of disjoint open intervals, $G$ is interpolated linearly.
 A  Lipschitz constant $V$ of $G$ is determined by the
rates $b,p$ in \eqref{generator}:
\begin{eqnarray*}\label{lipsch-const} V&=&2B\sum_{z\in\Z}|z|p(z),\qquad\mbox{ with }\\
B&=&\sup_{0\le a\le K,0\le k<
K}\{b(a,k)-b(a,k+1),b(k+1,a)-b(k,a)\} \end{eqnarray*}
This constant will be a key tool for the finite propagation property
of the particle system (see Proposition \ref{standard} below).
\\ \\
In \cite[Section 2.2]{bgrs2} we discussed the different definitions
of \emph{entropy solutions} (loosely speaking : the physically
relevant solutions)  to an equation such as \eqref{hydrodynamics},
and the ways to prove their existence and uniqueness. Therefore we
just briefly recall the definition of entropy solutions based on
shock admissibility conditions (Ole{\u\i}nik's entropy condition),
valid only for solutions with  locally bounded space-time variation
(\cite{vol}). A weak solution to \eqref{hydrodynamics} is an entropy
solution if and only if, for a.e. $t>0$, all discontinuities of
$u(.,t)$ are entropy shocks. A discontinuity $(u^-,u^+)$, with
$u^\pm:=u(x\pm 0,t)$, is called an entropy shock, if and only if:\\
\[
\ba{l} \mbox{The chord of the graph of $G$
between $u^-$ and $u^+$ lies}\\
\mbox{below the graph if $u^-<u^+$, above the graph if
$u^+<u^-$.}\ea\]
In the above condition, ``below" or ``above" is meant in wide
sense,  that is the graph and chord may
coincide at some points between $u^-$ and $u^+$. \\
 If the initial datum $u_0(.)$ has locally bounded
space-variation, there exists a unique entropy solution to
\eqref{hydrodynamics}, within the class of functions of locally
bounded space-time variation (\cite{vol}).
%
%
%

%
\subsection{Monotonicity and invariant measures}
Assumption {\em (A4)} implies the following monotonicity property,
 crucial in our approach. For the coordinatewise partial
order on $\bf X$, defined by $\eta\leq\zeta$ if and only if
$\eta(x)\leq\zeta(x)$ for every $x\in\Z$, we have
\be \label{attractive_1} (\eta_0,t)\mapsto\eta_t(\eta_0,\omega)
\mbox{ is nondecreasing w.r.t. }\eta_0 \ee
for every $\omega$ such that this mapping is well defined.
The partial order on $\bf X$ induces a partial stochastic order on
${\mathcal P}(\mathbf{X})$; namely, for  $\mu_1,\mu_2\in{\mathcal
P}(\mathbf{X})$,
we write $\mu_1\leq\mu_2$ if the following equivalent
conditions
hold  (see {\em e.g.} \cite{lig1}, \cite{strassen}):\\ \\
i) For every non-decreasing nonnegative function $f$ on $\mathbf
X$, $\mu_1(f)\leq\mu_2(f)$.\\
ii) There exists a coupling measure $\widetilde{\mu}$ on $\mathbf
{\widetilde X}=\mathbf {X}\times\mathbf {X}$ with marginals
$\mu_1$ and $\mu_2$, such that
$\widetilde{\mu}\{(\eta,\xi):\,\eta\leq\xi\}=1$.\\ \\
It follows from this and \eqref{attractive_1} that
\be \label{attractive_2} \mu_1\leq\mu_2\Rightarrow\forall
t\in\R^+,\,\mu_1 S(t)\leq\mu_2 S(t) \ee
Either property \eqref{attractive_1} or \eqref{attractive_2} is
usually called {\em attractiveness}.\\ \\
Let $\mathcal I$ and $\mathcal S$ denote respectively the set of
invariant probability measures for $L$, and the set of
shift-invariant probability measures on $\mathbf{X}$. We derived in
\cite[Proposition 3.1]{bgrs2}, that
\begin{proposition}\label{invariant}
\be \label{characterize_2} \left({\mathcal I}\cap{\mathcal
S}\right)_e= \left\{\nu^\rho,\,\rho\in{\mathcal R}\right\}\ee
 with $\mathcal R$ a closed subset of $[0,K]$ containing $0$
and $K$, and $\nu^\rho$ a shift-invariant measure such that
$\nu^\rho[\eta(0)]=\rho$.
(The index $e$ denotes extremal elements.)\\
The measures $\nu^\rho$ are stochastically ordered:
\be \label{ordered_measures}
\rho\leq\rho'\Rightarrow\nu^\rho\leq\nu^{\rho'}\ee
\end{proposition}
Moreover we have, quoting  \cite[Lemma 4.5]{rez}:
\begin{proposition}\label{Rezakhanlou}
The measure $\nu^\rho$ has {\em a.s.} density $\rho$, that is
\[
\lim_{l\to\infty}\frac{1}{2l+1}\sum_{x=-l}^l\eta(x)=\rho,
\quad \nu^\rho-\mbox{a.s.}
\]
\end{proposition}
%
%
%
%
%
%
By  \cite[Theorem 6]{kk}, \eqref{ordered_measures} implies existence
of a probability space on which one can define random variables
satisfying
\be\label{marginals}\eta^\rho\sim\nu^\rho\ee
and, with probability one,
\be \label{strassen} \eta^\rho\leq\eta^{\rho'},\, \forall
\rho,\rho'\in{\mathcal R}\mbox{ such that }\rho\leq\rho' \ee
Proceeding as in the proof of \cite[Theorem 7]{kko}, one can also
require (but we shall not use this property)  the joint distribution
of $(\eta^\rho:\rho\in{\mathcal R})$ to  be invariant by the spatial
shift $\tau_x$.
In the special case where $\nu^\rho$ are product measures, that is
when the function $b(.,.)$ satisfies assumptions of \cite{coc},
such a family of random variables can be constructed explicitely:
if $\left(U_x\right)_{x\in\Z}$ is a family of i.i.d. random
variables uniformly distributed on $(0,1)$, one defines
\be\label{product_coupling} \eta^\rho(x)=F_{\rho}^{-1}(U_x) \ee
where $F_\rho$ is the cumulative distribution function of the
single site marginal distribution of $\nu^\rho$. We will assume
without loss of generality (by proper enlargement) that the
``initial'' probability space $\Omega_0$ is large enough to define
a family of random variables satisfying
\eqref{marginals}--\eqref{strassen}.\\ \\
An important result for our approach is a {\em space-time} ergodic
theorem for particle systems mentioned in \cite{rez}, which we state
here in a general form,  and prove in Appendix
\ref{appendix_ergo}.
\begin{proposition}
\label{prop_ergo} Let $(\eta_t)_{t\ge 0}$ be a Feller process on
$\bf X$ with  a translation invariant generator $L$,  that is
\be\label{translation_gen}\tau_1 L\tau_{-1}=L\ee
Assume further that
\[
\mu\in({\mathcal I}_L\cap{\mathcal S})_e\]
where ${\mathcal I}_L$ denotes the set of invariant measures for
$L$. Then, for any local function $f$ on $\mathbf X$, and any $a>0$
\be \label{ergo}\lim_{\ell\to\infty}\frac{1}{a\ell^2}\int_0^{a\ell}
\sum_{i=0}^\ell\tau_i f(\eta_t)dt=\int f
d\mu=\lim_{\ell\to\infty}\frac{1}{a\ell^2}\int_0^{a\ell}
\sum_{i=-\ell}^{-1}\tau_i f(\eta_t)dt\ee
a.s. with respect to the law of the process with initial
distribution $\mu$.
\end{proposition}
\begin{remark}
\label{remark_ergo}
It follows from \eqref{ergo} that, more generally,
\be
\label{ergo2}\lim_{\ell\to\infty}\frac{1}{(b-a)(d-c)\ell^2}\int_{a\ell}^{b\ell}
\sum_{i\in\Z\cap[cl,dl]}\tau_i f(\eta_t)dt=\int f d\mu\ee
for every $0\leq a<b$ and $c<d$ in $\R$, as can be seen by
decomposing the space-time rectangle $[al,bl]\times[cl,dl]$ into
rectangles containing the origin.
\end{remark}
\section{Almost sure Riemann hydrodynamics}
\label{riemann}
By definition, the Riemann problem with values
$\lambda,\rho\in[0,K]$ for \eqref{hydrodynamics} is the Cauchy
problem  for the particular initial condition
\begin{equation}\label{eq:rie}
R_{\lambda,\rho}^0(x)=\lambda \mathbf 1_{\{x<0\}}+\rho \mathbf
1_{\{x\geq 0\}}
\end{equation}
The  entropy solution for this Cauchy datum will be denoted in the
sequel by $R_{\lambda,\rho}(x,t)$.
In this section, we derive the corresponding almost sure
hydrodynamic limit when $\lambda,\rho\in{\mathcal R}$.
We will use the following variational representation of the Riemann
problem. We henceforth assume $\lambda<\rho$ (for the case
$\lambda>\rho$, replace minimum with maximum below and make
subsequent changes everywhere in the section).
\begin{proposition}\label{proposition_1}
(\cite[Proposition 4.1]{bgrs2}).
Let $\lambda,\rho\in[0,K]$, $\lambda<\rho$. \\
i) There is a countable set $\Sigma_{low}\subset[\lambda,\rho]$
(depending only on the differentiability properties of the convex
envelope of $G$ on $[\lambda,\rho]$) such that, for every
$v\in\R\setminus\Sigma_{low}$, $G(\cdot)-v{\cdot}$ achieves its
minimum over $[\lambda,\rho]$ at a unique point $h_c(v)$. Then
$R_{\lambda,\rho}(x,t)=h_c(x/t)$ whenever $x/t\not\in\Sigma_{low}$. \\
ii) Suppose $\lambda,\rho\in{\mathcal R}$. Then the previous minimum
is unchanged if restricted to $[\lambda,\rho]\cap\mathcal R$. As a
result, the Riemann entropy solution is {\em a.e.} $\mathcal
R$-valued.
\end{proposition}
\begin{remark}
Property ii) holds if we replace $\mathcal R$ by any closed subset
of $[0,K]$ on the complement of which $G$ is affine. We stated it
for $\mathcal R$ because it is the set of densities relevant to the
particle system.
\end{remark}
To state Riemann hydrodynamics, we define a particular initial
distribution for the particle system. We introduce a transformation
$T:{\bf X}^2\to{\bf X}$ by
\[
T(\eta,\xi)(x)=\eta(x){\bf 1}_{\{x< 0\}}+\xi(x){\bf 1}_{\{x\geq
0\}} \]
We define $\nu^{\lambda,\rho}$ as the distribution of
$T(\eta^\lambda,\eta^\rho)=:\eta^{\lambda,\rho}$, and
$\bar{\nu}^{\lambda,\rho}$ as the coupling distribution of
$(\eta^\lambda,\eta^\rho)$. Note that, by \eqref{strassen},
\be \label{ordered_coupling} \bar{\nu}^{\lambda,\rho}\left\{
(\eta,\xi)\in{\bf X}^{2}:\,\eta\leq\xi
\right\}=1 \ee
The measure $\nu^{\lambda,\rho}$ is non-explicit unless we are in
the special case of \cite{coc} where the $\nu^\rho$ are product,
one can use \eqref{product_coupling}, and $\nu^{\lambda,\rho}$ is
itself a product measure. In all cases,
$\nu^{\lambda,\rho}$ enjoys the  properties:\\

P1) Negative (nonnegative) sites are distributed as under
$\nu^\lambda$ ($\nu^\rho$);

P2) $\tau_1\nu^{\lambda,\rho}\geq\nu^{\lambda,\rho}$
($\tau_1\nu^{\lambda,\rho}\leq\nu^{\lambda,\rho}$) if
$\lambda\leq\rho$ ($\lambda\geq\rho$);

P3) $\nu^{\lambda,\rho}$ is stochastically increasing with respect
to $\lambda$ and $\rho$.\\ \\
%
%
Let us also define an extended shift $\theta'$ on the compound
probability space $\Omega'={\bf X}^2\times\Omega$. This is a
particular case of $\widetilde{\Omega}$ when the ``initial''
probability space $\Omega_0$ is the set ${\bf X}^2$ of coupled
particle configurations $(\eta,\xi)$. Let
\be\label{omega'}\omega'=(\eta,\xi,\omega)\ee
denote a generic element of this space. We set
\be\label{extended_shift}\theta'_{x,t}\omega' =
(\tau_x\eta_t(\eta,\omega),\tau_x\eta_t(\xi,\omega),\theta_{x,t}\omega)
\ee
We can now state and prove the main results of this section.
\begin{proposition}
\label{proposition_2_2} Set
\be \label{def_N} {\mathcal N}^{v,w}_t(\omega'):=\sum_{[vt]<x\leq
[wt]}\eta_t(T(\eta,\xi),\omega)(x) \ee
Then, for every $t>0$,  $\alpha\in\R^+,\beta\in\R$ and
$v,w\in\R\setminus\Sigma_{low}$,
\begin{eqnarray} \nonumber &&\lim_{t\to\infty}\frac{1}{t}{\mathcal
N}^{v,w}_t\circ\theta'_{[\beta t],\alpha t}(\omega')\\&&\  =
[G(R_{\lambda,\rho}(v,1))-v\,R_{\lambda,\rho}(v,1)]-[G(R_{\lambda,\rho}(w,1))-w\,R_{\lambda,\rho}(w,1)]
\label{as}\end{eqnarray}
$\bar{\nu}^{\lambda,\rho}\otimes\Prob$-a.s.
\end{proposition}
\begin{remark}
This result is an almost sure version of  \cite[Lemma 3.2]{av},
where the limit of the corresponding expectation was derived.
\end{remark}
\begin{corollary}
\label{corollary_2_2} Set
\be \label{def_empirical_shift}
\beta^N_t(\omega')(dx):=\alpha^N(\eta_t(T(\eta,\xi),\omega))(dx)
\ee
(i) For every $t>0$,  $s_0\geq 0$ and $x_0\in\R$, we have the
$\bar{\nu}^{\lambda,\rho}\otimes\Prob$-a.s. convergence
\[
\lim_{N\to\infty}\beta^N_{Nt}(\theta'_{[Nx_0],Ns_0}\omega')(dx)=R_{\lambda,\rho}(.,t)dx
\]
(ii) In particular, for an initial sequence $(\eta_0^N)_N$ such
that $\eta_0^N=\eta^{\lambda,\rho}$ for every  $N\in\N$, the
conclusion of Theorem \ref{th:hydro} holds without the
finite-range assumption on $p(.)$.
\end{corollary}
For the asymmetric exclusion process, \cite{afs} proved a statement
equivalent to the particular case $\alpha=\beta=0$ of Proposition
\ref{proposition_2_2}. Their argument, which is a correction of
\cite{bf}, is based on subadditivity. As will appear in Section
\ref{general}, we do need to consider nonzero $\alpha$ and $\beta$
in order to prove a.s. hydrodynamics for general (non-Riemann)
Cauchy datum. Using arguments similar to those in \cite{afs}, it
would be possible to prove Proposition \ref{proposition_2_2} with
$\alpha=\beta=0$ for
our  
model \eqref{generator}. However this no longer works for nonzero
$\alpha$ and $\beta$ (see  Appendix \ref{remarks_sub}). Therefore
we construct a new  approach for 
a.s. Riemann hydrodynamics, that does not use subadditivity.\\ \\
To prove Proposition \ref{proposition_2_2}, we  first rewrite (in
Subsection \ref{subsubsection_currents}) the quantity ${\mathcal
N}^{v,w}_t$ in terms of particle currents for which we  then state
(in Subsection \ref{statement_lemmas}) a series of lemmas (proven in
Subsection \ref{proofs_lemmas}), and finally obtain the desired
limits for the currents, by deriving upper and lower bounds. For one
bound (the upper bound if $\lambda<\rho$ or lower bound if
$\lambda>\rho$), we  derive an ``almost-sure 
proof'' inspired by  the ideas of \cite{av} and their extension in
\cite{bgrs2}. For the other bound we use new ideas based on large
deviations of the empirical measure.
%
\subsection{Currents}
\label{subsubsection_currents}
Let us define particle currents in a system $(\eta_t)_{t\ge 0}$
governed by \eqref{update}--\eqref{noupdate}. Let
$x_.=(x_t,\,t\geq 0)$ be a $\Z$-valued {\em cadlag} path with
$\abs{x_t-x_{t^-}}\leq 1$. In the sequel this path will be either
deterministic, or a random path independent of the Poisson measure
$\omega$. We define the particle current as seen by an observer
travelling along this path. We first consider a semi-infinite
system, that is with $\sum_{x>0}\eta_0(x)<+\infty$: in this case, we
set
\be \label{def_current}
\varphi^{x_.}_t(\eta_0,\omega):=\sum_{y>x_t}\eta_t(\eta_0,\omega)
(y)-\sum_{y>x_0}\eta_0(y) \ee
 where $\eta_t(\eta_0,\omega)$ is the mapping
introduced in \eqref{unique_mapping}. In the sequel we shall most
often omit $\omega$ and write $\eta_t$ for $\eta_t(\eta_0,\omega)$
when this creates no ambiguity.
%
%
%
We have
\be \label{current_3} \varphi^{x_.}_t(\eta_0,\omega)
=\varphi^{x_.,+}_t(\eta_0,\omega)
-\varphi^{x_.,-}_t(\eta_0,\omega)+\tilde{\varphi}^{x_.}_t(\eta_0,\omega)
\ee
where
\be \label{plus_minus_current} \ba{lcl}
\varphi^{x_.,+}_t(\eta_0,\omega) &=& \omega \left\{(s,y,z,u):\,
0\leq s\leq t,\, y\leq
x_{s}<y+z,\phantom{\displaystyle\frac{(\eta_{s^-}}{||b||_\infty}}
\right.\\&&\left.\hskip 3cm\displaystyle
u\leq\frac{b(\eta_{s^-}(y),\eta_{s^-}(y+z))}{||b||_\infty}
\right\}\\
\displaystyle \varphi^{x_.,-}_t(\eta_0,\omega) &=& \omega \left\{
(s,y,z,u):\, 0\leq s\leq t,\, y+z\leq
x_{s}<y,\phantom{\displaystyle\frac{(\eta_{s^-}}{||b||_\infty}}
\right.\\&&\left.\hskip 3cm\displaystyle
u\leq\frac{b(\eta_{s^-}(y),\eta_{s^-}(y+z))}{||b||_\infty}
\right\} \ea \ee
%
%
%
count the number of rightward/leftward crossings due to particle
jumps, and
\be \label{self_motion}
\tilde{\varphi}^{x_.}_t(\eta_0,\omega)=-\int_{[0,t]}\eta_{s^-}(x_s\vee
x_{s^-})dx_s\ee
is the current due to the self-motion of the observer.
%
%
For an infinite system, we may still {\em define}
$\varphi^{x_.}_t(\eta_0,\omega)$ by equations \eqref{current_3} to
\eqref{self_motion}. We shall use the notation $\varphi^v_t$ in
the particular case 
 $x_t=[vt]$.
%
%
%
%
%
%
%
The following identities are immediate from \eqref{def_current} in
the semi-infinite case, and extend to the infinite case:
\beq \label{comparison_12}
\abs{\varphi^{x_.}_t(\eta_0,\omega)-\varphi^{y_.}_t(\eta_0,\omega)}&\leq&
K\left(\abs{x_t-y_t}+ \abs{x_0-y_0}\right)\\
\label{difference} \sum_{x=1+[vt]}^{[wt]}\eta_t(\eta_0,\omega)(x)
&=&\varphi^{v}_t(\eta_0,\omega)-\varphi^{w}_t(\eta_0,\omega) \eeq
%
%
%
Following \eqref{difference}, the quantity
 ${\mathcal N}_t^{v,w}(\omega')$ defined in
\eqref{def_N} can be written as
\be \label{difference_phi} {\mathcal
N}_t^{v,w}(\omega')=\phi^v_t(\omega')-\phi^w_t(\omega') \ee
where we define the current
\[
\phi^{v}_t(\omega'):=\varphi^{v}_t(T(\eta,\xi),\omega)\]
%
%
%
%
%
Notice that
\[
\phi^{v}_t(\eta,\eta,\omega)= \varphi^{v}_t(\eta,\omega)\]
The proof of the existence of the limit in Proposition
\ref{proposition_2_2} is thus reduced to
\be \label{lim_current}
\lim_{t\to\infty}\,t^{-1}\phi^{v}_t\left(\theta'_{[\beta t],\alpha
t}\omega'\right)\quad\mbox{exists }
\bar{\nu}^{\lambda,\rho}\otimes\Prob\mbox{-a.s.} \ee
\subsection{Lemmas}
\label{statement_lemmas}
We fix $\alpha\in\R^+$ and $\beta\in\R$. The first lemma deals with
equilibrium processes.
\begin{lemma}\label{lemma_2_2_2} For all $r\in[\lambda,\rho]\cap{\mathcal R}$,
$\varsigma\in\mathbf{X}$,   $v\in\R\setminus\Sigma_{low}$,
\[
\lim_{t\to\infty}\,t^{-1}\phi^{v}_t\circ\theta'_{[\beta t],\alpha
t}(\varsigma,\varsigma,\omega)= G(r)-vr,\quad
\bar{\nu}^{r,r}\otimes\Prob\mbox{-a.s.} \]
\end{lemma}
The second lemma relates the current of the process under study with
equilibrium currents; here it plays the role of Lemma 3.3 in
\cite{av}. It is connected to the ``finite propagation property'' of
the particle model (see
Lemma \ref{finite_prop}):
\begin{lemma}
\label{lemma_2_2_2_1} There exist $\bar{v}$ and $\underbar{v}$
(depending on $b$ and $p$) such that  we have,
$\bar{\nu}^{\lambda,\rho}\otimes{\Prob}\mbox{-a.s.}$,
\be \label{requilibrium_current}
\lim_{t\to\infty}\left[t^{-1}\phi^v_t\circ\theta'_{[\beta
t],\alpha t}(\eta,\xi,\omega)- t^{-1}\phi^v_t\circ\theta'_{[\beta
t],\alpha t}(\xi,\xi,\omega)\right]=0, \ee
for all $v> \bar{v}$, and
\be \label{lequilibrium_current}
\lim_{t\to\infty}\left[t^{-1}\phi^v_t\circ\theta'_{[\beta
t],\alpha t}(\eta,\xi,\omega)- t^{-1}\phi^v_t\circ\theta'_{[\beta
t],\alpha t}(\eta,\eta,\omega)\right]=0,\ee
for all $v <\underbar{v}$.
\end{lemma}
For the next lemmas we need some more notation and definitions.
Let  $v\in\R$. We consider a probability space $\Omega^+$, whose
generic element is denoted by $\omega^+$, on which is defined a
Poisson process $N_t=N_t(\omega^+)$ with intensity $\abs{v}$. We
denote by ${\Prob}^+$ the associated probability. We set
\beq\label{def_poisson}
x_s^t(\omega^+)&:=&({\rm sgn}(v))\left[N_{\alpha t+s}(\omega^+)-N_{\alpha t}(\omega^+)\right]\\
\label{mapping_tilde}\widetilde{\eta}^t_s(\eta_0,\omega,\omega^+)
&:=&\tau_{x_s^t(\omega^+)}\eta_s(\eta_0,\omega) \eeq
Thus $(\widetilde{\eta}_s^t)_{s\ge 0}$  is a Feller process
with generator
\be \label{new_gen} L_v=L+S_v,\quad S_v f(\zeta)=\abs{v}
[f(\tau_{{\rm sgn}(v)}\zeta)-f(\zeta)] \ee
(for $f$ local and $\zeta\in\bf X$) for which the set of local
functions is a core, as it is known to
be (\cite{lig1}) for  $L$. 
We denote by  ${\mathcal I}_v$ the set of invariant measures for
$L_v$. Since any translation invariant measure on $\bf X$ is
stationary for the pure shift generator $S_v$,  we have
\be\label{samesets}{\mathcal I}\cap{\mathcal S}={\mathcal
I}_v\cap{\mathcal S}\ee
It can be shown
(see \cite[Theorem 3.1]{f1}  and \cite[Corollary 9.6]{f2}) that
${\mathcal I}_v\subset {\mathcal S}$ for $v\neq 0$, which implies
$${\mathcal I}\cap{\mathcal S}={\mathcal
I}_v\cap{\mathcal S}={\mathcal I}_v $$
but we shall not use this fact here. Define the time empirical
measure
\be \label{def_empirical} m_{t}(\omega',\omega^+):=t^{-1}\int_0^t
\delta_{\widetilde{\eta}^t_s(T(\eta,\xi),\omega,\omega^+)}ds \ee
and  space-time empirical  measure (where $\eps>0$) by
\be \label{def_empirical_2}
m_{t,\eps}(\omega',\omega^+):=\abs{\Z\cap[-\eps t,\eps
t]}^{-1}\sum_{x\in\Z:\,\abs{x}\leq \eps
t}\tau_xm_{t}(\omega',\omega^+) \ee
We introduce the set
\[
{\mathcal M}_{\lambda,\rho}:=\left\{ \mu\in
{\mathcal P}({\bf X}):\,\nu^\lambda\leq\mu\leq\nu^\rho \right\}
\]
Notice that this is a closed (thus compact) subset of the compact
space ${\mathcal P}({\bf X})$.
\begin{lemma}\label{lemma_empirical}
(i) With
$\bar{\nu}^{\lambda,\rho}\otimes{\Prob}\otimes{\Prob}^+$-probability
one, every subsequential limit of   $m_{t,\eps }(\theta'_{[\beta
t],\alpha t}\omega',\omega^+)$
as $t\to\infty$ lies in   ${\mathcal I}_v\cap{\mathcal S}
\cap{\mathcal M}_{\lambda,\rho}={\mathcal I}\cap{\mathcal
S}\cap{\mathcal M}_{\lambda,\rho}$.\\
(ii) ${\mathcal I}\cap{\mathcal S}\cap{\mathcal M}_{\lambda,\rho}$
is the set of probability measures $\nu$ of the form
$\nu=\int\nu^r\gamma(dr)$,  where $\gamma$ is a probability measure
supported on ${\mathcal R}\cap[\lambda,\rho]$.
\end{lemma}
The proof of Lemma \ref{lemma_empirical} will be based on the
following large deviation result  in the spirit of \cite{dv1}.
\begin{lemma}\label{deviation_empirical}
 (i) The functional ${\mathcal D}_v$ defined on ${\mathcal
P}(\bf X)$ by
\be \label{def_dirichlet} {\mathcal D}_v(\mu):=\sup_{f\,\rm {local}}
-\int\frac{L_v e^f}{e^f}(\widetilde\eta)d\mu(\widetilde\eta) \ee
 is nonnegative,
lower-semicontinuous, and ${\mathcal D}_v^{-1}(0)={\mathcal I}_v$.\\ \\
(ii) Let $\widetilde{\xi}_.$ be a Markov process with generator
$L_v$ and distribution  denoted by  $\bf P$. Define the
empirical  measures
\be \label{def_empirical-gen}
\pi_t(\widetilde\xi_.):=t^{-1}\int_0^t\delta_{\widetilde\xi_s}ds,\quad
\pi_{t,\eps}:=\abs{\Z\cap[-\eps t,\eps
t]}^{-1}\sum_{x\in\Z\cap[-\eps t,\eps t]}\tau_x\pi_t
\ee
where $\eps>0$.  Then, for every closed   subset
$F$ of ${\mathcal P}(\bf X)$ and every $t\geq 0$,
\be \label{ld} \limsup_{t\to\infty}t^{-1}\log{\bf
P}\left(\pi_{t,\eps}(\widetilde\xi_.)\in F\right)\leq -\inf_{\mu\in
F}{\mathcal D}_v(\mu) \ee
\end{lemma}
\subsection{Proofs of Lemma \ref{lemma_2_2_2},
Proposition \ref{proposition_2_2} and Corollary
\ref{corollary_2_2}} \label{proof_prop}
\begin{proof}{Proposition}{proposition_2_2}
We denote by $\omega'=(\eta,\xi,\omega)$ a generic element of
$\Omega'$. We will establish the following limits: first,
%
%
%
\beq \inf_{r\in[\lambda,\rho]\cap{\mathcal R}} [G(r)-v r] & \leq
 & \liminf_{t\to\infty}\,t^{-1}\phi^{v}_t\circ\theta'_{[\beta
t],\alpha t}(\omega')\nonumber\\
& \leq &
\limsup_{t\to\infty}\,t^{-1}\phi^{v}_t\circ\theta'_{[\beta
t],\alpha t}(\omega')\nonumber\\
& \leq & \sup_{r\in[\lambda,\rho]\cap{\mathcal R}} [G(r)-vr],\quad
\bar{\nu}^{\lambda,\rho}\otimes\Prob\mbox{-a.s.}\label{liminf_better}
\eeq
and then
\be \label{limsup}
\limsup_{t\to\infty}\,t^{-1}\phi^{v}_t\circ\theta'_{[\beta
t],\alpha t}(\omega')\leq
\inf_{r\in[\lambda,\rho]\cap{\mathcal R}} [G(r)-vr],\quad
\bar{\nu}^{\lambda,\rho}\otimes\Prob\mbox{-a.s.} \ee
which will imply the result, when combined with Proposition
\ref{proposition_1} and  the expression \eqref{difference_phi} of
${\mathcal N}_t^{v,w}(\omega')$.
Though only the first inequality in \eqref{liminf_better} (together
with \eqref{limsup})  seems relevant to Proposition
\ref{proposition_2_2}, we will need the whole set of inequalities:
Indeed, writing \eqref{liminf_better} for
$\lambda=\rho=r\in{\mathcal R}$  proves the equilibrium result
of Lemma \ref{lemma_2_2_2}.\\ \\
To obtain the bounds in \eqref{liminf_better} we proceed as follows.
First we replace the deterministic path $vt$ in the current
$\phi^{v}_t$ by  $x_t^t(\omega^+)$.  Then we consider a spatial
average of  $\varphi^{x^t_.(\omega^+)+x}_t$  for $x\in[-\eps t,\eps
t]$,  and we introduce, for $\zeta\in\mathbf{X}$, the martingale
$M^{x,v}_t(\zeta,\omega,\omega^+)$  associated to
$\varphi^{x^t_.(\omega^+)+x}_t(\zeta,\omega)$  (see below
\eqref{martingale}). An exponential bound on the martingale reduces
the derivation of \eqref{liminf_better} to  bounds (deduced thanks
to Lemma \ref{lemma_empirical}) on
\[
\int [f(\eta)-v\eta(1)]m_{t,\eps}(\theta'_{[\beta t],\alpha
t}\omega',\omega^+)(d\eta) \]
(see below \eqref{replace_oncemore}), where $[f(\eta)-v\eta(1)]$
corresponds to the compensator of
$\varphi^{x^t_.(\omega^+)+x}_t(\zeta,\omega)$ in
$M^{x,v}_t(\zeta,\omega,\omega^+)$.
The bound \eqref{limsup} relies on Lemmas \ref{lemma_2_2_2} and
\ref{lemma_2_2_2_1} combined with the monotonicity of the
process.\\ \\
{\em Step one: proof of \eqref{liminf_better}.}
 We have
\be \label{sothat} t^{-1}\phi^v_t\circ\theta'_{[\beta t],\alpha
t}(\omega')=t^{-1}\varphi^v_t(\varpi_{\alpha t},\theta_{[\beta
t],\alpha t}\omega) \ee
where  the configuration
\be \label{def_varpi} \varpi_{\alpha t}=\varpi_{\alpha
t}(\eta,\xi,\omega):=T\left( \tau_{[\beta t]}\eta_{\alpha
t}(\eta,\omega), \tau_{[\beta t]}\eta_{\alpha t}(\xi,\omega)
 \right)
\ee
depends only on the restriction of $\omega$ to $[0,\alpha
t]\times\Z$. Thus, since $\omega$ is a Poisson measure under
$\Prob$, $\theta_{[\beta t],\alpha t}\omega$ is independent of
$\varpi_{\alpha t}(\eta,\xi,\omega)$ under
$\bar{\nu}^{\lambda,\rho}\otimes\Prob$, and
\be \label{hence_varpi} \mbox{under
$\bar{\nu}^{\lambda,\rho}\otimes\Prob\otimes\Prob^+,\,\varpi_{\alpha
t}$ is independent of $(\theta_{[\beta t],\alpha
t}\omega,\omega^+)$}\ee
 Define
$$
\psi_t^{v,\eps}(\zeta,\omega,\omega^+):=\abs{\Z\cap[-\eps t,\eps
t]}^{-1}\sum_{y\in\Z:\,\abs{y}\leq\eps
t}\varphi^{x^t_.(\omega^+)+y}_t(\zeta,\omega)
$$
for every $(\zeta,\omega,\omega^+)\in{\bf
X}\times\Omega\times\Omega^+$, with $x^t_.(\omega^+)$ given by
\eqref{def_poisson}.
By \eqref{comparison_12},
$$
\abs{\varphi^v_t(\zeta,\omega)-\psi_t^{v,\eps}(\zeta,\omega,\omega^+)
 }\leq
K\left(2\eps t+\abs{x^t_t(\omega^+)-vt}\right) $$
Since $t^{-1}x_t^t(\omega^+)\to v$ with $\Prob^+$-probability one,
the proof of \eqref{liminf_better} can be reduced to that of the
same inequalities with the l.h.s. of \eqref{sothat} replaced by
\be \label{replace_1} t^{-1}\psi^{v,\eps}_t(\varpi_{\alpha
t},\theta_{[\beta t],\alpha t}\omega,\omega^+) \ee
and $\bar{\nu}^{\lambda,\rho}\otimes\Prob$ replaced by
$\bar{\nu}^{\lambda,\rho}\otimes\Prob\otimes\Prob^+$.
Let $f(\eta):=f^+(\eta)-f^-(\eta)$, with
\beq\nonumber
f^+(\eta) &=& \sum_{y,z\in\Z:\,y\leq
0<y+z}p(z)b(\eta(y),\eta(y+z))\\\nonumber f^-(\eta) & = &
\sum_{y,z\in\Z:\,y+z\leq 0<y}p(z)b(\eta(y),\eta(y+z)) \eeq
and define $\delta$ to be $1$ if $v>0$, $0$ if $v<0$, and any
integer if $v=0$. By  the definition of particle current
\eqref{current_3}--\eqref{self_motion}, we have that, for any
$\zeta\in\mathbf{X}$,
\beq M^{x,v}_t(\zeta,\omega,\omega^+)&:=&
\varphi^{x^t_.(\omega^+)+x}_t(\zeta,\omega)\nonumber\\
&& -\int_0^t
[f(\tau_x\widetilde{\eta}^t_{s^-}(\zeta,\omega,\omega^+))\nonumber\\
&&
\qquad-v\widetilde{\eta}^t_{s^-}(\zeta,\omega,\omega^+)(x+\delta)]
ds \label{martingale}\\
E^{x,v,\theta}_t(\zeta,\omega,\omega^+) & := & \exp\left\{
\theta\varphi^{x^t_.(\omega^+)+x}_t(\zeta,\omega)\right.
\nonumber\\
&&\left.- \left( e^\theta-1\right)\int_0^t
f^+(\tau_x\widetilde{\eta}^t_{s^-}(\zeta,\omega,\omega^+))ds\right.
\nonumber\\
&& \left. - \left( e^{-\theta}-1\right)\int_0^t
f^-(\tau_x\widetilde{\eta}^t_{s^-}(\zeta,\omega,\omega^+))ds\right.\nonumber\\
&& \left. -\int_0^t \abs{v}\left( e^{-{\rm
sgn}(v)\theta\widetilde{\eta}^t_{s^-}(\zeta,\omega,\omega^+)(x+\delta)}
-1 \right) ds
\right\}\label{girsanov}\\
& = & \exp\left\{ \theta M^{x,v}_t(\zeta,\omega,\omega^+)\right.
\nonumber\\
&&\left.- \left( e^\theta-1-\theta \right)\int_0^t
f^+(\tau_x\widetilde{\eta}^t_{s^-}(\zeta,\omega,\omega^+))ds\right.\nonumber\\
&& \left. - \left( e^{-\theta}-1+\theta \right)\int_0^t
f^-(\tau_x\widetilde{\eta}^t_{s^-}(\zeta,\omega,\omega^+))ds
\right.\nonumber\\
&&- \int_0^t \abs{v}\left( e^{-{\rm
sgn}(v)\theta\widetilde{\eta}^t_{s^-}(\zeta,\omega,\omega^+)(x+\delta)}
-1\right.
\nonumber\\
&&\left.\left.\qquad\qquad+{\rm
sgn}(v)\theta\widetilde{\eta}^t_{s^-}(\zeta,\omega,\omega^+)(x+\delta)
\right)ds \right\}\nonumber \eeq
 are martingales under $\Prob\otimes\Prob^+$, with respective
means $0$ and $1$. Notice that $\eta_{s^-}$ and $\eta_s$ can be
replaced with each other in the above martingales, because,  by the
graphical construction of Section \ref{graphical_construction},
$s\mapsto \eta_s(x)$ is $\Prob\otimes{\Prob}^+$-a.s. locally
piecewise constant for every $\zeta\in{\bf X}$ and $x\in\Z$. It
follows from \eqref{girsanov} that
\be \label{expo_bound} {\Exp}\left( e^{ \theta M^{x,v}_t }
\right)\leq e^{ Ct \left( e^{K\abs{\theta}} -1-K\abs{\theta} \right)
} \ee
 for any $\zeta\in{\bf X}$, where expectation is w.r.t.
$\Prob\otimes\Prob^+$, and the constant $C$ depends only on $p(.)$,
$b(.)$ and $v$ but not on $\zeta$. Cramer's inequality and
\eqref{expo_bound} imply the large deviation bound
\be \label{deviation} \Prob\otimes\Prob^+( \{ \abs{ M^{x,v}_t }\geq
y \} )\leq 2e^{ -t{\mathcal I}_C(y) } \ee
 for any $\zeta\in{\bf X}$ and $y\geq 0$, with the rate
function
\[
{\mathcal I}_C(y)=\frac{y}{K}\log\left( 1+\frac{y}{CK}
\right)-C\left[ \frac{y}{CK}-\log\left( 1+\frac{y}{CK} \right)
\right]
\]
Because of the independence property \eqref{hence_varpi}, by
\eqref{deviation},
$$
\bar{\nu}^{\lambda,\rho}\otimes{\Prob}\otimes{\Prob}^+\left(
\{ \abs{ M^{x,v}_t( \varpi_{\alpha t},\theta_{[\beta t],\alpha
t}\omega,\omega^+) ) }\geq y\}
\right)\leq 2e^{ -t{\mathcal I}_C(y) }
$$ 
 This and Borel Cantelli's lemma imply that
\be \label{as_2} \lim_{t\to\infty}t^{-1}\abs{\Z\cap[-\eps t,\eps
t]}^{-1}\sum_{x\in\Z:\,\abs{x}\leq\eps t}M^{x,v}_t( \varpi_{\alpha
t},\theta_{[\beta t],\alpha t}\omega,\omega^+) =0,
 \ee
$\bar{\nu}^{\lambda,\rho}\otimes{\Prob}\otimes{\Prob}^+\mbox{-a.s.}$
In view of \eqref{martingale} and \eqref{as_2}, the proof of
\eqref{liminf_better} is now reduced to that of the same set of
inequalities with \eqref{replace_1} replaced by
\begin{eqnarray*} t^{-1}\abs{\Z\cap[-\eps t,\eps
t]}^{-1}\int_0^t\sum_{x\in\Z:\,\abs{x}\leq\eps t}&&[\tau_x
f(\widetilde{\eta}^t_{s}(\varpi_{\alpha t},\theta_{[\beta t],\alpha
t}\omega,\omega^+)\nonumber\\
&&\, -\tau_x v\widetilde{\eta}^t_{s}(\varpi_{\alpha
t},\theta_{[\beta t],\alpha t}\omega,\omega^+)(1)]ds
\end{eqnarray*}
 which is exactly, because of \eqref{def_varpi},
\be \label{replace_oncemore} \int
[f(\eta)-v\eta(1)]m_{t,\eps}(\theta'_{[\beta t],\alpha
t}\omega',\omega^+)(d\eta) \ee
with the empirical measure $m_{t,\eps}$ defined in
\eqref{def_empirical_2}.
By Proposition \ref{invariant} and Lemma \ref{lemma_empirical},
every subsequential limit  $\nu$ as $t\to\infty$ of
$m_{t,\eps}(\theta'_{[\beta t],\alpha t}\omega',\omega^+)$ is of the
form
$$
\nu=\int\nu^r\gamma(dr)
$$
for some measure $\gamma$  supported on  ${\mathcal
R}\cap[\lambda,\rho]$. Then the corresponding subsequential limit as
$t\to\infty$ of \eqref{replace_oncemore} is of the form
$$
\int[G(r)-vr]\gamma(dr)
$$
 as one verifies, using shift invariance of $\nu^r$, that
$$
\int [f(\eta)-v\eta(1)]\nu^r(d\eta)=G(r)-vr
$$
This concludes the proof.\\ \\
{\em Step two: proof of \eqref{limsup}.}
Let $u_1\le \underbar{v}\le v\le \bar v\le v_1$,
$\lambda<\rho\in\mathcal R$, and $r\in[\lambda,\rho]\cap{\mathcal
R}$. We set $\varsigma=\eta^r$, and we define
$\bar{\nu}^{\lambda,r,\rho}$ as the coupling distribution of
$(\eta^\lambda,\eta^r,\eta^\rho)$. Note that, by the stochastic
 ordering property \eqref{strassen},
\be \label{ordered_coupling_trio}
\bar{\nu}^{\lambda,r,\rho}\left\{
(\eta,\varsigma,\xi)\in{\bf X}^{3}:\,\eta\leq\varsigma \leq\xi
\right\}=1 \ee
The following limits are true for
$\bar{\nu}^{\lambda,r,\rho}$-a.e. $(\eta,\varsigma,\xi)$. By
 the expression \eqref{difference_phi} of ${\mathcal
N}^{.,.}_t$ and the equilibrium limit Lemma \ref{lemma_2_2_2},
\be \label{eq:step1} \lim_{t\to\infty}\,\frac{1}{t}{\mathcal
N}^{u_1,v}_t\circ\theta'_{[\beta t],\alpha
t}(\varsigma,\varsigma,\omega)=r(v-u_1)\ee
By attractiveness,
\be\label{eq:step1bis}
\liminf_{t\to\infty}\frac{1}{t}
{\mathcal N}^{u_1,v}_t\circ\theta'_{[\beta t],\alpha
t}(\varsigma,\xi,\omega) \ge
\lim_{t\to\infty}\,\frac{1}{t}
{\mathcal N}^{u_1,v}_t\circ\theta'_{[\beta t],\alpha
t}(\varsigma,\varsigma,\omega) \ee
Putting together \eqref{eq:step1} and \eqref{eq:step1bis}, \be
\label{eq:step2} \liminf_{t\to\infty}\frac{1}{t}{\mathcal
N}^{u_1,v}_t\circ\theta'_{[\beta t],\alpha
t}(\varsigma,\xi,\omega)\ge r(v-u_1)\ee
Now, by \eqref{difference_phi}, Lemma \ref{lemma_2_2_2_1} and
Lemma \ref{lemma_2_2_2} respectively for $r$ and $\rho$,
\be\label{eq:step3} \lim_{t\to\infty}\,\frac{1}{t}{\mathcal
N}^{u_1,v_1}_t\circ\theta'_{[\beta t],\alpha
t}(\varsigma,\xi,\omega) =(G(r)-u_1 r)-(G(\rho)-v_1\rho)\ee
Subtracting \eqref{eq:step2} to \eqref{eq:step3}, we get
\be\label{eq:step4} \limsup_{t\to\infty}\,\frac{1}{t}{\mathcal
N}^{v,v_1}_t\circ\theta'_{[\beta t],\alpha
t}(\varsigma,\xi,\omega) \le(G(r)-vr)-(G(\rho)-v_1\rho)\ee
By attractiveness, \eqref{ordered_coupling_trio} and
\eqref{eq:step4}, we have
\beq\nonumber \limsup_{t\to\infty}\,\frac{1}{t}{\mathcal
N}^{v,v_1}_t\circ\theta'_{[\beta t],\alpha t}(\eta,\xi,\omega)
&\le & \limsup_{t\to\infty}\,\frac{1}{t}{\mathcal
N}^{v,v_1}_t\circ\theta'_{[\beta t],\alpha t}(\varsigma,\xi,\omega)\\
&\le & (G(r)-vr)-(G(\rho)-v_1\rho) \label{eq:step5} \eeq
 Using \eqref{difference_phi}, \eqref{eq:step5}, Lemma
\ref{lemma_2_2_2_1}, and Lemma \ref{lemma_2_2_2} for $\rho$, we
obtain
\begin{eqnarray*}\label{eq:step6} &&\limsup_{t\to\infty}\,
\frac{1}{t}\phi_t^v\circ\theta'_{[\beta t],\alpha
t}(\omega')\\&=& \limsup_{t\to\infty}\,
\left(\frac{1}{t}\phi_t^v\circ\theta'_{[\beta t],\alpha
t}(\omega')
-\frac{1}{t}\phi_t^{v_1}\circ\theta'_{[\beta t],\alpha
t}(\omega')
+\frac{1}{t}\phi_t^{v_1}\circ\theta'_{[\beta t],\alpha
t}(\omega')\right)\\ &\leq&
\limsup_{t\to\infty}\,\frac{1}{t}{\mathcal
N}^{v,v_1}_t\circ\theta'_{[\beta t],\alpha t}(\omega')
+\limsup_{t\to\infty}\,\frac{1}{t}\phi_t^{v_1}\circ\theta'_{[\beta
t],\alpha t}(\omega')\\ &\leq&
((G(r)-vr)-(G(\rho)-v_1\rho))+(G(\rho)-v_1\rho)\\
& = & G(r)-vr\end{eqnarray*}
for every $r\in[\lambda,\rho]\cap{\mathcal R}$. Since $\varsigma$ is
no longer involved in the above inequalities, we obtain a
$\bar{\nu}^{\lambda,\rho}$-a.s. limit with respect to $(\eta,\xi)$
for every $r\in[\lambda,\rho]\cap{\mathcal R}$. By continuity of $G$
this holds outside a common exceptional set  of
$\bar{\nu}^{\lambda,\rho}$-probability $0$  for all
$r\in[\lambda,\rho]\cap{\mathcal R}$. This proves \eqref{limsup}.
\end{proof}
\mbox{}\\
\begin{proof}{Corollary}{corollary_2_2}
\\ \\
{\em (i).} By Proposition \ref{proposition_1}  (cf.
\cite[(28)]{bgrs2}),
$$
\frac{d}{dv}[G(h_c(v))-vh_c(v)]=-h_c(v)
$$
weakly with respect to $v$.  Thus, setting
$u=R_{\lambda,\rho}$, we have
\be \label{difference_2}
[G(u(v,1))-vu(v,1)]-[G(u(w,1))-wu(w,1)]=\int_v^w u(x,1)dx \ee
for all $v,w\in\R$.  Let $a<b$ in $\R$. Setting
$$
\varpi=T\left(\tau_{[Nx_0]}\eta_{Ns_0}(\eta,\omega),\tau_{[Nx_0]}
\eta_{Ns_0}(\xi,\omega)\right)
$$
we have
\begin{eqnarray*}
\beta^N_{Nt}(\theta'_{[Nx_0],Ns_0}(\eta,\xi,\omega))((a,b]) & = &
N^{-1}\sum_{[Na]<x\leq [Nb]}\eta_{Nt}(\varpi,\theta_{[Nx_0],Ns_0}
\omega)(x) \\
& = & t (Nt)^{-1}{\mathcal
N}^{a/t,b/t}_{Nt}(\theta'_{[Nx_0],Ns_0}\omega')
\end{eqnarray*}
Thus, by Proposition \ref{proposition_2_2} and
\eqref{difference_2},
\be
\lim_{N\to\infty}\beta^N_{Nt}(\theta'_{[Nx_0],Ns_0}\omega')((a,b])
 = t\int_{a/t}^{b/t}u(x,1)dx
 = \int_a^b u(x,t)dx\label{lim_beta}\ee
$\bar{\nu}^{\lambda,\rho}\otimes\Prob$-a.s., where the last
equality follows from Proposition \ref{proposition_1}.
Now \eqref{difference_2} implies that the r.h.s. of \eqref{as} is a
continuous function of $(v,w)$, while the l.h.s. is a uniformly
Lipschitz function of $(v,w)$, since the number of particles per
site is bounded. It follows that one can find a single exceptional
set  of $\bar{\nu}^{\lambda,\rho}\otimes\Prob$-probability $0$
outside which \eqref{lim_beta} holds simultaneously for all
$a,b$, which proves the claim.\\ \\
 {\em (ii).} Since $\eta^{\lambda,\rho}$ has distribution
$\bar{\nu}^{\lambda,\rho}$ under $\Prob_0$, the statement follows
from {\em (i)} with $x_0=s_0=0$.
\end{proof}
\subsection{Proofs of remaining lemmas}\label{proofs_lemmas}
\begin{proof}{Lemma}{lemma_2_2_2_1}
Let $\eps>0$. We consider the probability space
$\Omega'\times(\Z^+)^\Z$ equipped with the product measure
$$
\Prob'_\eps:=\bar{\nu}^{\lambda,\rho}\otimes\Prob\otimes P_\eps
$$
where $P_\eps$ is the product measure on $\Z$ whose marginal at
each site is Poisson with mean $K(1+\eps)$.
A generic element of this space is denoted by $(\omega',\chi)$,
with $\omega'=(\eta,\xi,\omega)$ and  $\chi\in(\Z^+)^\Z$.
We first prove \eqref{requilibrium_current}. Because of
\eqref{ordered_coupling}  (that is, coupled configurations are
ordered under $\bar{\nu}^{\lambda,\rho}$), by the attractiveness
property \eqref{attractive_1}, we may define
\be \label{def_gamma}
\gamma_s(\omega'):=\eta_s(\xi,\omega)-\eta_s(T(\eta,\xi),\omega)
\ee
for $s\geq 0$.
 Therefore $\gamma$-particles represent the discrepancies
between the system starting from $\xi$ and the system starting
from $T(\eta,\xi)$. We look for $\bar v$ such that there are no
discrepancies to the right of $\bar v t$, in which case the
particles there should be distributed like $\xi$-particles,
according to the equilibrium measure $\nu^\rho$.
Let $v>\bar v$. Because $\gamma_{0}(x)=0$ for all $x>0$, by
 the definition of current \eqref{def_current},
\[
\phi^v_t(\xi,\xi,\omega)=\phi^v_t(\omega')+
\sum_{y>[vt]}\gamma_{t}(y) \]
Therefore, to prove \eqref{requilibrium_current}, we want to
obtain
\be \label{therefore_53} \lim_{t\to\infty}t^{-1}\sum_{y>\bar v
t}\gamma_t\circ\theta'_{[\beta t],\alpha
t}(\omega')(y)=0,\quad\bar{\nu}^{\lambda,\rho}\otimes\Prob\mbox{-a.s.}
\ee
To this end, we follow the proof of   \cite[Proposition 5]{afs},
with minor modifications.  We emphasize that even if the latter
proof corresponds to $\alpha=\beta=0$, we will see that the
arguments extend to $(\alpha,\beta)\neq (0,0)$. We label
$\gamma$-particles with $R_.^.$'s and $\chi$-particles with
$Z_.^.$'s as follows: we denote by $R_0^j=R_0^j(\omega')$, resp.
$Z_0^j=Z_0^j(\chi)$, the position of the $\gamma_0$-particle, resp.
$\chi$-particle, with label $j$
 (we take $j\le 0$).
The labelling is such that $R_0^j\le R_0^{j+1}$ and $Z_0^j\le
Z_0^{j+1}$ for all  $j<0$, where $R_0^{0}$, resp. $Z_0^0$, is the
position of the first $\gamma_0$-particle, resp. $\chi$-particle,
to the left of (or at) site $0$.  By the definition of $\chi$, the
number of $\chi$-particles between $-n$ and $0$ will be eventually
larger than $nK$. Let
$$
W(\chi):=\inf\left\{
n\in\Z^+:\,Z^j_0\geq -[\abs{j}/K]\mbox{ for every }j\leq -n
\right\}
$$
By Poisson large deviation bounds, the random variable $W$ is
$\Prob'_\eps$-a.s. finite with exponentially decaying
distribution.
Since $\gamma_0(x)\leq K$ for every $x\in\Z$, we have
$R^j_0\leq-[\abs{j}/K]$ for all $j\leq 0$, hence $Z_0^j\ge R_0^j$
for every   $j\leq -W(\chi)$.  The dynamics of $Z^j_.$ is
defined 
by: if $R_{t_-}^j=x$ and, for some $z>0$  and $u\in[0,1]$,
\be \label{evolution_z}
\{(t,x,z,u),(t,x+z,-z,u)\}\cap\,\omega\neq\emptyset\ee
then $Z_t^j=Z_{t_-}^j+z$.
In other words, $\chi$-particles evolve as mutually independent
(given their initial positions) random walks, that jump from $y$
to $y+z$ at rate
\[\overline{p}(z)=(p(z)+p(-z))||b||_\infty\]
for all $y\in\Z, z\ge 0$. Then, since a jump for $R_{.}^j$ from
$R_{t_-}^j=x$ to $R_{t}^j=x+z$ is possible only under
\eqref{evolution_z},
\be \label{obvious_coupling} R_0^j\leq Z_0^j\Rightarrow \forall
t>0,\,R_t^j\leq Z_t^j \ee
 In view of \eqref{therefore_53},
\eqref{obvious_coupling}, we estimate
\beq
\sum_{x>\bar v t}\gamma_{t}(x)=\sum_{j\le 0} \mathbf
1_{\{R_{t}^j(\omega')>\bar v t\}}  & = & \sum_{j\le -W(\chi)}
\mathbf 1_{\{R_{t}^j(\omega')>\bar v t\}}+
\sum_{-W(\chi)<j\leq 0} \mathbf
1_{\{R_{t}^j(\omega')>\bar v t\}}\nonumber\\
& \leq & \sum_{j\le -W(\chi)} \mathbf
1_{\{Z_{t}^j(\omega')>\bar v t\}}+W(\chi)\nonumber\\
& \leq &  \bar Z_{t}(\omega,\chi)+W(\chi)\label{disc} \eeq
 where $\bar Z_{t}:=\sum_{i\le 0}\mathbf
1_{\{Z_{t}^i>\bar v t\}}$ is a Poisson random variable with mean
\be\label{exp_Z} \Exp'_\eps\bar Z_{t}= K(1+\eps)\sum_{j\ge \bar v
t} \Prob'_\eps(Y_{t}>j)\ee
for $Y_{t}$ a random walk starting at $0$ that jumps from $y$ to
$y+z$ with rate $\overline{p}(z)$. Repeating \cite[(43)--(45)]{afs}
 gives that
\be\label{repeat_afs}\lim_{t\to\infty}\Exp'_\eps \bar Z_{t}/t=0
\ee
if we choose $\bar v>\sum_{z>0}z\overline{p}(z)$.
Let $\delta>0$. Since $\bar Z_{t}$ is a Poisson variable, \beq
\Prob'_\eps(|\bar Z_{t}-\Exp'_\eps\bar Z_{t}|>\delta t)&\le &
\frac{\Exp'_\eps(\bar Z_{t}-\Exp'_\eps\bar Z_{t})^4}
{(\delta t)^4}\nonumber\\
&\le & \frac{[\Exp'_\eps(\bar Z_{t}-\Exp'_\eps\bar Z_{t})^2]^2}
{(\delta t)^4}\nonumber\\
&\le & \frac{(\Exp'_\eps\bar Z_{t}/t)^2t^2}{(\delta
t)^4}\label{tail_estimate_Z} \eeq
Therefore, by Borel Cantelli lemma
\be\label{borel_cantelli_Z}\lim_{t\to\infty}(\bar
Z_{t}-\Exp'_\eps\bar Z_{t})/t=0,\quad\Prob'_\eps\mbox{-a.s.}\ee
Since $\Prob$ is invariant by $\theta_{[\beta t],\alpha t}$, $\bar
Z_t(\theta_{[\beta t],\alpha t}\omega,\chi)$ has the same
distribution as $\bar Z_t(\omega,\chi)$ under $\Prob'_\eps$.
 Thus \eqref{exp_Z}--\eqref{borel_cantelli_Z} still hold with
$\bar Z_t(\theta_{[\beta t],\alpha t}\omega,\chi)$ instead of
$\bar Z_t(\omega,\chi)$, and
\be \label{as_Z} \lim_{t\to\infty}t^{-1}\bar Z_t(\theta_{[\beta
t],\alpha t}\omega,\chi)=0,\quad\Prob'_\eps\mbox{-a.s.} \ee
 Because the random variable $W$ in \eqref{disc} does not
depend on $\omega'$, \eqref{as_Z} and \eqref{disc} imply
\eqref{therefore_53}. This concludes the proof of
\eqref{requilibrium_current}.\\ \\
If we now define $\gamma_s$ as
\[
\gamma_s(\omega'):=\eta_s(T(\eta,\xi),\omega)-\eta_s(\eta,\omega)
\]
 and replace $\phi^v_t(\omega')$ by $-\phi^v_t(\omega')$
(so that the current, which was rightwards, becomes leftwards),
then the proof of \eqref{lequilibrium_current} can be obtained by
repeating the same steps as in the previous argument.
\end{proof}
\mbox{}\\ \\
 \begin{proof}{Lemma}{lemma_empirical}
\mbox{}\\ \\
\textbf{(i).}
Since
$$
m_t(\omega',\omega^+)-m_{[t]}(\omega',\omega^+)=\frac{[t]-t}{t}m_{[t]}
(\omega',\omega^+)+
t^{-1}\int_{[t]}^t\delta_{\widetilde{\eta}^t_s(T(\eta,\xi),\omega,\omega^+)}ds
$$
has total  variation  bounded by $2/t$, it is enough to prove the
result for every subsequential limit of the sequence
$m_{n,\eps}(\theta'_{[\beta n],\alpha n}\omega',\omega^+)$ as
$n\to\infty$, $n\in\N$.\\ \\
{\em Step one.} We prove that every subsequential limit lies in
${\mathcal I}_v$. It is enough to show that, for every  open
neighborhood $O$ of ${\mathcal I}_v$, with
$\bar{\nu}^{\lambda,\rho}\otimes\Prob\otimes\Prob^+$-probability
one, $m_{n,\eps}$ lies in $O$ for sufficiently large $n$.
 One can see from  \eqref{def_empirical_2}  and
\eqref{def_empirical-gen} that
$$m_{t,\eps}(\theta'_{[\beta t],\alpha
t}\omega',\omega^+)=\pi_{t,\eps}(\widetilde{\xi}^{^t}_.)$$
where, for fixed $t$, $\widetilde{\xi}^t_.$ is the process defined
by
$$
\widetilde{\xi}^{t}_s:=\widetilde{\eta}^t_s(\varpi_{\alpha
t},\theta_{[\beta t],\alpha t}\omega,\omega^+)
$$
with the configuration $\varpi_{\alpha t}$ defined in
\eqref{def_varpi} and satisfying \eqref{hence_varpi}. Hence under
$\bar{\nu}^{\lambda,\rho}\otimes\Prob\otimes\Prob^+$,
$\widetilde{\xi}^t_.$  is a Markov process with generator $L_v$
 and initial distribution $\nu^{\lambda,\rho}$ independent of
$t$.
By  Lemma  \ref{deviation_empirical},
\[
\limsup_{n\to\infty}n^{-1}\log\bar{\nu}^{\lambda,\rho}\otimes\Prob\otimes\Prob^+
\left( m_{n,\eps}(\theta'_{[\beta n],\alpha
n}\omega',\omega^+)\not\in O\right)<0 \]
Now Borel-Cantelli's lemma implies that, a.s. with respect to
$\bar{\nu}^{\lambda,\rho}\otimes\Prob\otimes\Prob^+$,
 $m_{n,\eps}(\theta'_{[\beta n],\alpha n}\omega',\omega^+)\in
O$  for large $n$.
\\ \\
{\em Step two.} We prove that every subsequential limit lies in
${\mathcal M}_{\lambda,\rho}$.  Since
$$
\eta\leq T(\eta,\xi)\leq\xi
$$
 for $\bar{\nu}^{\lambda,\rho}$- a.e. $(\eta,\xi)$,
\eqref{def_varpi} and
 the monotonicity property \eqref{attractive_1} imply
\[
\tau_{[\beta t]}\eta_{\alpha t}(\eta,\omega)\leq\varpi_{\alpha
t}(\eta,\xi,\omega)\leq \tau_{[\beta t]}\eta_{\alpha t}(\xi,\omega)
\]
for all $t\geq 0$.  By \eqref{def_empirical} and \eqref{def_varpi},
$$m_t(\theta'_{[\beta t],\alpha
t}\omega',\omega^+)=t^{-1}\int_0^t\delta_{\widetilde{\eta}^t_s(\varpi_{\alpha
t},\theta_{[\beta t],\alpha t}\omega,\omega^+)}ds$$
By \eqref{mapping_markov}--\eqref{mapping_shift}  and
\eqref{def_poisson}--\eqref{mapping_tilde},
$$
\widetilde{\eta}^t_s(\tau_{[\beta t]}\eta_{\alpha
t}(\eta,\omega),\theta_{[\beta t],\alpha t}\omega,\omega^+)
=\tau_{-\sgn(v)N_{\alpha t}(\omega^+)+[\beta t ]}
\widetilde{\eta}^0_{\alpha t+s}(\eta,\omega,\omega^+)
$$
Thus
\beq & & l_t(\eta,\omega,\omega^+):=t^{-1}\int_{\alpha t}^{\alpha
t+t}\abs{\Z\cap[-\eps t,\eps t]}^{-1}\sum_{x\in{\mathcal
X}_t(\omega^+)}\tau_x
\delta_{\widetilde{\eta}^0_{s}(\eta,\omega,\omega^+)}ds\nonumber\\
& \leq & m_{t,\eps}(\theta'_{[\beta t],\alpha
t}\omega',\omega^+)\nonumber\\
& \leq & t^{-1} \int_{\alpha t}^{\alpha t+t}\abs{\Z\cap[-\eps t,\eps
t]}^{-1}\sum_{x\in{\mathcal
X}_t(\omega^+)}\tau_x\delta_{\widetilde{\eta}^0_{s}(\xi,\omega,\omega^+)}ds=:r_t(\xi,\omega,\omega^+)\label{sandwich_empirical}
\eeq
$\bar{\nu}^{\lambda,\rho}\otimes\Prob\otimes\Prob^+$- a.s. for all
$t\geq 0$, where
$$
{\mathcal X}_t(\omega^+):=\Z\cap\left[[\beta t]-\sgn(v)N_{\alpha
t}(\omega^+)-\eps t,[\beta t]-\sgn(v)N_{\alpha t}(\omega^+)+\eps
t\right]
$$
We now argue that $l_t$ and $r_t$ respectively converge a.s. to
$\nu^\lambda$ and $\nu^\rho$ with respect to
$\bar{\nu}^{\lambda,\rho}\otimes\Prob\otimes\Prob^+$. Let us
consider for instance $l_t$.
Let $L_t$ denote the measure defined as $l_t$ but with ${\mathcal
X}_t(\omega^+)$ replaced by
$${\mathcal Y}_t:=\Z\cap[(\beta-\alpha v-\eps)t,(\beta-\alpha
v+\eps)t]$$
By the strong law of large numbers for Poisson processes, there
exists a subset ${\mathcal C}\subset\Omega^+$, with
$\Prob^+$-probability one, on which  $\sgn(v)N_t(\omega^+)/t\to v$
as $t\to\infty$. The total variation of $l_t-L_t$ is bounded by
$(2\varepsilon t+1)^{-1}\abs{{\mathcal X}_t\Delta{\mathcal Y}_t}$,
where $\Delta$ denotes the symmetric difference of two sets. Hence,
for $\omega^+\in\mathcal C$, $l_t-L_t$ converges to $0$ in total
variation. We are thus reduced to proving a.s. convergence of $L_t$
to $\nu^\lambda$.
Under $\overline{\nu}^{\lambda,\rho}\otimes\Prob\otimes\Prob^+$,
$\widetilde{\eta}^0_.(\eta,\omega,\omega^+)$ is a Feller process
with generator $L_v$ and initial distribution $\nu^\lambda$. Thanks
to \eqref{characterize_2} and \eqref{samesets}, we can apply
Proposition \ref{prop_ergo}, or more precisely, its extended form
\eqref{ergo2}. This implies convergence of $L_t$.  \\ \\
{\em Step three.} We prove that every subsequential limit lies in
${\mathcal S}$. To this end we simply note that the measure
\begin{eqnarray} \tau_1
m_{n,\eps}-m_{n,\eps}&=&\left|[-\eps n,\eps
n]\cap\Z\right|^{-1}\left(\sum_{x\in\Z\cap(-\eps n+1,\eps
n+1]}\tau_x
m_n\right.\nonumber\\
&&\left.\phantom{\ \left|[-\eps n,\eps n]\cap\Z\right|^{-1}\
}-\sum_{x\in\Z\cap[-\eps n,\eps n)}\tau_x
m_n\right)\label{shift_difference}
\end{eqnarray}
has total  variation  bounded by $2\abs{[-\eps n,\eps
n]\cap\Z}^{-1}=O(1/n)$. Letting $n\to\infty$ in
\eqref{shift_difference} shows that $\tau_1 m=m$ for any
subsequential limit $m$ of $m_{n,\eps}$.\\ \\
\textbf{(ii).} Proposition \ref{invariant} implies
$\nu=\int\nu^r\gamma(dr)$  with $\gamma$ supported on $\mathcal R$.
Let $[\lambda',\rho']\subset{\mathcal R}$ denote the support of
$\gamma$, and assume for instance that $\lambda'<\lambda$. Choose
some $\lambda''\in(\lambda',\lambda)$. By Proposition
\ref{Rezakhanlou}, the random variable
$$
M(\eta):=\lim_{l\to\infty}(2l+1)^{-1}\sum_{x=-l}^l\eta(x)
$$
is defined $\nu^r$-a.s. for every $r\in{\mathcal R}$, and thus also
$\nu$-a.s. It is a nondecreasing function of $\eta$. Thus,
$\nu^\lambda\leq\nu$ implies
$$
\nu(M<\lambda'')\leq\nu^\lambda(M<\lambda'')=0
$$
where the last equality follows from Proposition
\ref{Rezakhanlou}, hence a contradiction. Similarly $\rho'>\rho$
implies a contradiction. Thus $\gamma$ is supported on ${\mathcal
R}\cap[\lambda,\rho]$.
\end{proof}
\mbox{}\\
\begin{proof}{Lemma}{deviation_empirical}
\mbox{}\\ \\
 {\em (i).}
Nonnegativity follows from taking $f=0$ in \eqref{def_dirichlet}. As
a supremum of continuous functions, ${\mathcal D}_v$ is lower
semicontinuous : indeed, because  the number of particles per site
is bounded,  each local function $e^f$ is continuous and bounded and
so is $L_v(e^f)/e^f$, hence the functional defined on ${\mathcal
P}({\bf X})$ by
\be\label{phi_f} \phi_f(\mu)=-\int \frac{L_v(e^{f})}{e^f}d\mu \ee
is continuous. The inclusion ${\mathcal I}_v\subset{\mathcal
D}_v^{-1}(0)$ holds because of
\[
L_v(\log g)\leq L_vg/g\]
which follows from the elementary inequality $\log b -\log a\leq
(b-a)/a$, and the fact that $\int L_v(\log g)d\mu=0$ if
$\mu\in{\mathcal I}_v$. We eventually prove the reverse inclusion
${\mathcal D}_v^{-1}(0)\subset{\mathcal I}_v$. Fix a local test
function  $f$.  If $\mu\in{\mathcal D}_v^{-1}(0)$, we must have
\be \label{wemusthave} I(t):=\int
\frac{L_v(e^{tf})}{e^{tf}}d\mu\geq 0,\quad\forall t\in\R\ee
As $f$ is  local  and the space $\{0,\ldots,K\}$ is finite,
integrability conditions are satisfied to differentiate $I(t)$ in
\eqref{wemusthave} under the integral. Since $I(0)=0$, equation
\eqref{wemusthave} implies that $I(t)$ has a minimum at $t=0$,
hence
$$
0=\frac{dI(t)}{dt}_{|t=0}=\int\frac{d}{dt}\frac{L_v(e^{tf})}
{e^{tf}}_{|t=0}d\mu=\int L_vf d\mu
$$
and thus $\mu\in{\mathcal I}_v$, since this is true for any local
function.\\ \\
{\em (ii).}
Since $\phi_f$ is continuous, by \cite[Appendix 2, Lemma 3.3]{kl},
it is enough to prove that
\be\label{ld_open} \limsup_{t\to\infty}t^{-1}\log{\bf
P}\left(\pi_{t,\eps}(\widetilde{\xi}_.)\in O\right)\leq
\inf_{f\,\mbox{\tiny local}}\sup_{\mu\in O}-\phi_f(\mu) \ee
for every open subset $O\subset{\mathcal P}({\bf X})$.
Let $f$ be a local test function on $\bf X$,
and set
$$
\bar{f}(t,\eta):=\abs{\Z\cap[-\eps t,\eps t]}^{-1} \sum_{x\in
\Z\cap[-\eps t,\eps t]} \tau_x f(\eta)=\sum_{n=0}^\infty
1_{[n\varepsilon^{-1},(n+1)\varepsilon^{-1})}(t)\bar{f}_n(\eta)
$$
where
$$
\bar{f}_n(\eta):=(2n+1)^{-1}\sum_{x=-n}^n\tau_x f(\eta)
$$
For each $n\in\Z^+$,
$$M^{f,n}_t:=\exp\left\{
\bar{f}_n(\widetilde\xi_t)-\bar{f}_n(\widetilde\xi_{n\varepsilon^{-1}})-
\int_{n\varepsilon^{-1}}^t
e^{-\bar{f}_n} L_v[e^{\bar{f}_n}] (\widetilde\xi_{s^-})ds
\right\},\quad t\geq n\varepsilon^{-1} $$
is a mean $1$ martingale under $\bf P$ with respect to the
$\sigma$-field ${\mathcal G}_t$ generated by
$(\widetilde\xi_s,\,s\leq t)$ (cf. \cite[Section 7 of Appendix
1]{kl}).
It follows that $M^f_t$ defined for $t\geq 0$ by
$$
M^f_t:=\prod_{k=1}^{n}M^{f,k-1}_{(k\varepsilon^{-1})^-}\,M^{f,n}_t,\quad
t\in[n\varepsilon^{-1},(n+1)\varepsilon^{-1})
$$
(where the product is $1$ for $n=0$) is a mean $1$ ${\mathcal
G}_t$-martingale under $\bf P$. Thus we can define a probability
measure ${\bf P}^f$ on ${\mathcal G}_t$ by $d{\bf P}^f/d{\bf
P}=M_t^f$.
A simple computation shows that
\be\label{expomart} M^f_t=\exp\left\{
\bar{f}(t,\widetilde\xi_t)-\bar{f}(0,\widetilde\xi_0)-\int_0^t
e^{-\bar{f}}L_v[e^{\bar{f}}](s,\widetilde\xi_{s^-})ds+R^f_t
\right\} \ee
where
\be\label{remainder} R^f_t=\sum_{n=1}^{[\eps
t]}\left[\bar{f}_{n-1}(\widetilde{\xi}_{(n\varepsilon^{-1})^-})-
\bar{f}_{n}(\widetilde{\xi}_{n\varepsilon^{-1}})\right] \ee
Notice that, by the graphical construction of Section
\ref{graphical_construction},  $s\mapsto\eta_s(x)$ is for each
$x\in\Z$ a piecewise constant function whose jumps occur at (random)
times which are a subset of some Poisson process.  Thus $s^-$ in
\eqref{expomart} can be replaced by $s$, and $(n\varepsilon^{-1})^-$
in \eqref{remainder} by $n\varepsilon^{-1}$. The latter implies that
the summand in \eqref{remainder} is bounded in modulus by
$4(2n+1)^{-1}\sup\abs{f}$. Hence
\be\label{bound_remainder} \abs{R^f_t}\leq 2[1+\log(\eps
t)]\sup\abs{f}\ee
We claim (this will be established below) that, for every
probability measure $\mu$ on $\bf X$, the mapping $f\mapsto\int
e^{-f}L_v[e^f]d\mu$ (defined on the set of local functions $f:{\bf
X}\to\R$) is convex. Since $L_v$ commutes with the space shift, this
implies
%
%
%
%
%
%
%
%
%
$$
\frac{d{\bf P}}{d{\bf P}^f}\leq\exp\left\{ -R_t^f+
\bar{f}(0,\widetilde{\xi}_0)-\bar{f}(t,\widetilde{\xi}_t)
+ t\int
\frac{L_v[e^f]}{e^f}(\eta)\pi_{t,\eps}(\widetilde{\xi}_.)(d\eta)\right\}
$$
 Thus, for any open subset $O$ of ${\mathcal P}({\bf X})$, we
have (cf. \eqref{phi_f})
\begin{eqnarray} {\bf P}(\pi_{t,\eps}(\widetilde\xi_.)\in
O)&\leq& e^{-R_t^f+2\sup\abs{f}}\int
e^{-t\phi_f[\pi_{t,\eps}(\widetilde\xi_.)]}{\bf
1}_O[\pi_{t,\eps}(\widetilde\xi_.)]d{\bf P}^f(\widetilde\xi_.)\nonumber\\
&\leq& \exp\left\{-R_t^f+2\sup\abs{f}-t\inf_{\mu\in O}\phi_f(\mu)
\right\}\label{noting_that}
\end{eqnarray}
Using \eqref{bound_remainder} and minimizing the
r.h.s. of \eqref{noting_that} over local functions $f$, we obtain \eqref{ld_open}.\\ \\
{\em Proof of claim.} We prove that $f\mapsto \int
e^{-f}L_v[e^f]d\mu=-\phi_f(\mu)$ is convex for any $\mu\in{\mathcal
P}({\bf X})$. Equivalently we show that, for any  local functions
$f,g$ on $\bf X$,  $t\mapsto -\phi_{tf+g}(\mu)$  is convex on $\R$.
We have
$$
\frac{d^2}{dt^2}\int\frac{L_v e^{tf+g}}{e^{tf+g}}d\mu=\int\frac{
L_v(f^2 e^{tf+g})
-2(tf+g)L_v(f e^{tf+g})
+f^2L_v(e^{tf+g})}{e^{tf+g}
}d\mu
$$
The above integrand is nonnegative. Indeed, for local functions
$\varphi$ and $\psi$,
\be\label{carreduchamp} L_v (\varphi^2\psi)-2 \varphi
L_v(\varphi\psi)+\varphi^2 L_v\psi=
L_v^\psi(\varphi^2)-2\varphi L_v^\psi\varphi
\ee
where $L_v^\psi\varphi:=L_v(\varphi\psi)-\varphi L_v\psi$. For
$\psi\geq 0$, $L_v^\psi$ is a Markov generator, and thus the r.h.s.
of \eqref{carreduchamp} is nonnegative.
\end{proof}
\section{The Cauchy problem}
\label{general}
%
%
In Corollary \ref{corollary_2_2}, we established an almost sure
hydrodynamic limit for initial measures corresponding to the Riemann
problem with $\mathcal R$-valued initial densities. In this section
we prove that this implies Theorem \ref{th:hydro}, that is the
almost sure hydrodynamic limit for {\em any} initial sequence
associated with {\em any} measurable initial density profile  (thus
we add in this Section the hypothesis $p(.)$ finite range, which was
not necessary for Riemann a.s.  hydrodynamics).
 This passage is inspired by Glimm's scheme, a well-known
procedure in the theory of hyperbolic conservation laws, by which
one constructs general entropy solutions using only Riemann
solutions  (see {\em e.g.} \cite[Chapter 5]{serre}).
In  \cite[Section 5]{bgrs2}, we undertook such a derivation for
convergence in probability. In the present context of almost sure
convergence, new error analysis is necessary. In particular, we have
to do an explicit time discretization (vs. the ``instantaneous
limit'' of \cite[Section 3, Theorem 3.2]{bgrs1} or \cite[Section
5]{bgrs2} for the analogue of \eqref{wearegoing} below), we need
estimates uniform in time (Lemma \ref{cor_approx}), and each
approximation step requires a control with
exponential bounds (Proposition \ref{lemma_bm} and Lemma \ref{finite_prop}). 
%
%
\subsection{Preliminary results}
For two measures $\alpha,\beta\in{\mathcal M}^+(\R)$ with compact
support, we define
\be\label{def_delta}\Delta(\alpha,\beta):=\sup_{x\in\R}\abs{
\alpha((-\infty,x])-\beta((-\infty,x])}\ee
When $\alpha$ or $\beta$  
is of the form $u(.)dx$ for $u(.)\in L^\infty(\R)$
 with compact support, we simply
write $u$ in \eqref{def_delta} instead of $u(.)dx$.
A connection between $\Delta$ and vague convergence is given by the
following technical lemma, whose proof is left to the reader.
\begin{lemma}
\label{delta_conv}
(i)  Let $(\alpha^N)_N$ be a ${\mathcal M}^+(\R)$-valued sequence
supported on a common compact subset of $\R$,  and $u(.)\in
L^\infty(\R)$. The following statements are equivalent: (a)
$\alpha^N\to u(.)dx$ as $N\to\infty$; (b) $\Delta(\alpha^N,u(.))\to 0$ as $N\to\infty$.\\ \\
(ii) Let $(\alpha^N(.))_N$ be a sequence   of $\mathcal
M^+(\R)$-valued functions  $\alpha^N:{\mathcal T}\to{\mathcal
M}^+(\R)$,  where $\mathcal T$ is an arbitrary set, such that the
measures $\alpha^N(t)$ are supported on a common compact subset of
$\R$. Assume that, for some $\alpha:[0,+\infty)\to {\mathcal
M}^+(\R)$, $\Delta(\alpha^N(t),\alpha(t))$ converges to $0$
uniformly on $\mathcal T$.  Then $\alpha^N(.)$ converges to
$\alpha(.)$ uniformly on $\mathcal T$.
\end{lemma}
The following proposition is a collection of results on entropy
solutions. We first recall two definitions.
A sequence $(u_n,n\in\N)$ of Borel measurable functions on $\R$ is
said to converge to $u$ in $L^1_{\rm loc}(\R)$ if and only if
$$
\lim_{n\to\infty} \int_I\abs{u_n(x)-u(x)}dx=0
$$
for every bounded interval $I\subset\R$.
The variation of a function $u(.)$ on an interval $I\subset\R$ is
defined by
$$
{\rm TV}_I[u(.)]= \sup\left\{
\sum_{i=0}^{n-1} \abs{ u(x_{i+1})-u(x_i) }:\,
n\in\N,\,x_0,\ldots,x_n\in I, x_0<\cdots<x_n
\right\}
$$
We shall simply write ${\rm TV}$ for ${\rm TV}_\R$.
We say that $u=u(.,.)$ defined on $\R\times\R^{+*}$ has locally
bounded space variation if  for every bounded   space
interval $I\subset\R$ and bounded time interval $J\subset\R^{+*}$
\[
\sup_{t\in J} {\rm TV}_I[u(.,t)]<+\infty \]
\begin{proposition}
\label{standard}
\mbox{}\\ \\
o) Let $u(.,.)$ be the entropy solution to \eqref{hydrodynamics}
with Cauchy datum  $u_0\in L^\infty(\R)$. Then the mapping
$t\mapsto u(.,t)$ lies in $C^0([0,+\infty),L^1_{\rm loc}(\R))$.\\
\\
i) If $u_0(.)$ is {\em a.e.} $\mathcal R$-valued, then so is the
corresponding entropy solution
$u(.,t)$ to \eqref{hydrodynamics} at later times.\\ \\
ii) If $u_0^i(.)$ has finite variation, that is ${\rm
TV}u_0^i(.)<+\infty$, then so does $u^i(.,t)$ for every $t>0$, and
${\rm TV}u^i(.,t)\leq{\rm TV}u_0^i(.)$.\\ \\
iii) Finite propagation property: Assume $u^i(.,.)$ ($i\in\{1,2\}$)
is the entropy solution to \eqref{hydrodynamics} with Cauchy data
$u_0^i(.)$. Let $V=||G'||_\infty:=\sup_\rho\abs{G'(\rho)}$. Then:
(a) for every $x<y$ and $0\leq t<(y-x)/2V$,
\be \label{lonestab}
\int_{x+Vt}^{y-Vt}\abs{u^1(z,t)-u^2(z,t)}dz\leq\int_x^y\abs{u_0^1(z)-u_0^2(z)}dz
\ee
In particular, if  $u_0^1$ is supported (resp. coincides with
$u_0^2$) in $[-R,R]$ for some $R>0$, $u^1(.,t)$ is supported (resp.
coincides with $u^2(.,t)$) in $[-R-Vt,R+Vt]$. (b) If $\int_\R
u_0^i(z)dz<+\infty$,
\be \label{deltastab}
\Delta(u^1(.,t),u^2(.,t))\leq\Delta(u^1_0(.),u^2_0(.)) \ee
iv) Let $x_0=-\infty<x_1<\cdots<x_n<x_{n+1}=+\infty$ and
$\eps:=\min_{0\le k \le n}(x_{k+1}-x_k)$.  Denote by $u_0(.)$ the
piecewise constant profile with value $r_k$ on $I_k:=(x_k,x_{k+1})$.
Then, for $t<\eps/(2V)$, the entropy solution $u(.,t)$ to
\eqref{hydrodynamics} with Cauchy datum $u_0(.)$ is given by
$$
u(x,t)=R_{r_{k-1},r_k}(x-x_k,t),\quad \forall
x\in\left(x_{k-1}+Vt,x_{k+1}-Vt\right)
$$
\end{proposition}
Properties o), ii) and iii) are standard, see  \cite{kru, serre,
lax1}. Properties i) and iv) are respectively \cite[Lemma
5.3]{bgrs2} and \cite[Lemma 3.4]{bgrs1}. The latter states that the
entropy solution starting from a piecewise constant profile can be
constructed at small times as a superposition of successive
non-interacting Riemann waves. This is a consequence of iii). Note
that the whole space is indeed covered by the definition of $u(x,t)$
in iv), since we have $x_{k+1}-Vt\geq x_k+Vt$ for $t\leq
\eps/(2V)$.\\ \\
The next lemma   improves \cite[Lemma 5.5]{bgrs2} by deriving an
approximation uniform in time.
%
%
\begin{lemma}
\label{cor_approx} Assume $u_0(.)$ is a.e. $\mathcal R$-valued, has
bounded support and finite variation, and let $(x,t)\mapsto u(x,t)$
be the entropy solution to \eqref{hydrodynamics} with Cauchy datum
$u_0(.)$. For every $\eps>0$,  let ${\mathcal P}_\eps$ be the set of
piecewise constant $\mathcal R$-valued functions on $\R$ with
compact support and step lengths at least $\eps$, and  set
\[
\delta_\eps(t):=\eps^{-1}\inf\{ \Delta(u(.),u(.,t)):\,
u(.)\in{\mathcal P}_\eps \}
\]
Then there is a sequence $\eps_n\downarrow 0$ as $n\to\infty$ such
that $\delta_{\eps_n}$ converges to $0$ uniformly on any bounded
subset of $\R^+$.
\end{lemma}
\begin{proof}{Lemma}{cor_approx}
We first argue that, for every $\eps>0$, $\delta_\eps$ is a
continuous function. Indeed,
 by Proposition \ref{standard}, iii),a) for every $T>0$,
there exists a bounded set $K_T\subset\R$ such that the support of
$u(.,t)$ is contained in $K_T$ for every $t\in[0,T]$. Since
\be\label{delta_inequality}
\Delta(v,w)\leq\int_\R\abs{v(x)-w(x)}dx\ee
 for $v,w\in L^\infty(\R)$ with compact support, it follows by
Proposition \ref{standard}, o)  and Lemma \ref{delta_conv},
(i)  that
\be\label{lim-delta} \lim_{s\to t}\Delta(u(.,s),u(.,t))=0 \ee
for every $t\geq 0$. This and the inequality
\[
\abs{ \delta_\eps(t)-\delta_\eps(s)
}\leq\eps^{-1}\Delta(u(.,s),u(.,t))
\]
imply continuity of $\delta_\eps$.
By Proposition \ref{standard}, i) and ii), $u(.,t)$ has bounded,
finite space variation, and is $\mathcal R$-valued. Hence, by
\cite[Lemma 5.5]{bgrs2},  for any given $\delta>0$,  for $\eps>0$
small enough, there exists an approximation
$u^{\eps,\delta}(.,t)\in{\mathcal P}_\eps$ of $u(.,t)$ with
%
$\Delta\left(u^{\eps,\delta}(.,t),u(.,t)\right)\leq \eps\delta $.
This implies $\delta_\eps(t)\to 0$ as $\eps\to 0$ for every $t>0$.
Let $\mathcal T$ be some countable dense subset of $[0,+\infty)$. By
the diagonal extraction procedure, we can find a sequence
$\eps_n\downarrow 0$ such that $\delta_{\eps_n}(t)\downarrow 0$ for
each $t\in{\mathcal T}$. By continuity of $\delta_{\eps}$ we also
have that $\delta_{\eps_n}(t)\downarrow 0$ for every
$t\in[0,+\infty)$. Dini's theorem implies that $\delta_{\eps_n}$
converges uniformly to $0$ on every bounded subset of $[0,+\infty)$.
\end{proof}
\\ \\
We now quote  \cite[Proposition 3.1]{bm}, which yields that
$\Delta$ is an ``almost'' nonincreasing functional for two coupled
particle systems:
\begin{proposition}
\label{lemma_bm}
Assume $p(.)$ is finite range. Then there exist constants $C>0$ and
$c>0$, depending only on $b(.,.)$ and $p(.)$, such that the
following holds. For every  $N\in\N$, $(\eta_0,\xi_0)\in{\bf X}^2$
with
$\abs{\eta_0}+\abs{\xi_0}:=\sum_{x\in\Z}[\eta_0(x)+\xi_0(x)]<+\infty$,
 and every $\gamma>0$, the event
\be\label{bmevent} \forall t>0:
\,\Delta(\alpha^N(\eta_t(\eta_0,\omega)),\alpha^N(\eta_t(\xi_0,\omega)))
\leq \Delta(\alpha^N(\eta_0),\alpha^N(\xi_0))+\gamma
\ee
has $\Prob$-probability at least
$1-C(\abs{\eta_0}+\abs{\xi_0})e^{-cN\gamma}$.
%
%
\end{proposition}
%
%
%
%
%
  We finally recall the {\em finite propagation} property at
particle level (see \cite[Lemma 5.2]{bgrs2}), which is a microscopic
analogue of Proposition \ref{standard}, iii).
\begin{lemma}
\label{finite_prop} There exist constants $v$ and $C$, depending
only on $b(.,.)$ and $p(.)$, such that the following holds. For any
$x,y\in\Z$, any $(\eta_0,\xi_0)\in{\bf X}^2$, and any
$0<t<(y-x)/(2v)$: if $\eta_0$ and $\xi_0$ coincide on the site
interval $[x,y]$, then  with $\Prob$-probability at least
$1-e^{-Ct}$, $\eta_s(\eta_0,\omega)$ and $\eta_s(\xi_0,\omega)$
coincide on the site interval $[x+vt,y-vt]\cap\Z$ for every
$s\in[0,t]$.
\end{lemma}
\begin{remark}\label{remark_finiteprop}
 The time uniformity in Proposition \ref{lemma_bm} and Lemma
\ref{finite_prop} does not appear in the original statements
(repectively, \cite[Proposition 3.1]{bm} and \cite[Lemma
5.2]{bgrs2}), but follows in each case from the proof.
\end{remark}
%
%
%
%
\subsection{Proof of Theorem \ref{th:hydro}}
\subsubsection{Simplified initial conditions}
\label{proof_simplified}
We will first prove Theorem \ref{th:hydro}  under the simplifying
assumptions:
\be\label{finite_0}u_0  \mbox{ is a.e. } {\mathcal
R}\mbox{-valued}\ee
\be\label{finite_00} {\rm TV} u_0  <  +\infty
\ee
and there exists $R>0$ (independent of $N$) such that
\be\label{finite_1} {\rm supp}\,u_0 \subset [-R,R]\ee
\be\label{finite_2} \forall
N\in\N\mbox,\quad\Prob_0\left(\eta^N_0(x)=0\mbox{ whenever
}x\in\Z,\,\abs{x}\geq RN\right)=1\ee
%
%
%
The essential part of the work  (that is, the approximation scheme)
is contained here, and the proof under general assumptions
will follow in  Subsection \ref{General case}  by approximation arguments.\\ \\
Assumption \eqref{finite_0} implies by Proposition \ref{standard},
i), that $u(.,t)$ is $\mathcal R$-valued, \eqref{finite_00} by
Proposition \ref{standard}, ii), that $u(.,t)$ has finite variation
for $t>0$, and \eqref{finite_1} by Proposition \ref{standard}, iii),
that $u(.,t)$ is supported on $[-(R+Vt),R+Vt]$. In the sequel we
abbreviate, for $(\omega_0,\omega)\in\widetilde\Omega$
$$
\eta^N_t=\eta_{t}(\eta^N_0(\omega_0),\omega)
$$
We consider the random process on $\widetilde{\Omega}$
\[
\Delta^N(t):=\Delta(\alpha^N(\eta^N_{Nt}),u(.,t))
\]
%
By  initial assumption \eqref{initial_profile_vague} and (i) of
Lemma \ref{delta_conv}, $\Delta^N(0)$ converges to $0$,
$\Prob_0$-a.s. Fix an arbitrary time $T>0$. We are going to prove
that
\be \label{wearegoing}
\lim_{N\to\infty}\sup_{t\in[0,T]}\Delta^N(t)=0,\quad\widetilde{\Prob}\mbox{-a.s.}
\ee
 Then one can find a set of probability one on which this
holds simultaneously for all $T>0$. Theorem \ref{th:hydro} follows
from (ii) of
Lemma \ref{delta_conv}.\\ \\
Let $\eps=\eps_n$ be given by Lemma \ref{cor_approx}, and
$\delta=\delta_n= 2 \sup_{t\in[0,T]}\delta_{\eps_n}(t)$,  so that
$\delta_n\to 0$ as $n\to\infty$. In the sequel, for notational
simplicity, we omit mention of $n$.
We fix a time discretization step
\be \label{time_step} \eps'=\eps\min((2v)^{-1},(2V)^{-1}) \ee
where $V$ and $v$ are the constants defined in Proposition
\ref{standard} and Lemma \ref{finite_prop}. Let $t_k=k\eps'$, for
 $k\le {\mathcal K}:=[T/\eps']$, $t_{{\mathcal K}+1}=T$. The
main step to derive \eqref{wearegoing} is to obtain a time
discretized version  of it, namely

\begin{lemma}\label{wearegoing_1}
\[
\limsup_{N\to\infty}\sup_{k=0,\ldots,{\mathcal
K}-1}\left[\Delta^N(t_{k+1})-\Delta^N(t_k)\right]\leq
3\delta\eps,\quad\widetilde{\Prob}\mbox{-a.s.} \]
\end{lemma}
The second, more technical step, will be to fill in the gaps between
discretized times by a uniform estimate for the time modulus of
continuity, that is

\begin{lemma}\label{wearegoing_2}
\[
\lim_{\eps=\eps_n\to
0}\limsup_{N\to\infty}\sup_{k=0,\ldots,{\mathcal
K}}\sup_{t\in[t_k,t_{k+1}]}\Delta\left[
\alpha^N(\eta^N_{Nt}),\alpha^N(\eta^N_{Nt_k})\right]=0,
\quad\widetilde{\Prob}\mbox{-a.s.} \]
\end{lemma}
By o) of Proposition \ref{standard},
$t\mapsto u(t,.)$ is uniformly continuous from $[0,T]$ to $L^1_{\rm
loc}(\R)$. This and \eqref{delta_inequality} imply an analogue of
Lemma \ref{wearegoing_2} at the level of entropy solutions, namely
\be\label{wearegoing_3} \lim_{\eps=\eps_n\to
0}\sup_{k=0,\ldots,{\mathcal
K}}\sup_{t\in[t_k,t_{k+1}]}\Delta(u(.,t),u(.,t_k))=0\ee
Then \eqref{wearegoing} follows from  Lemma \ref{wearegoing_1},
Lemma \ref{wearegoing_2} and \eqref{wearegoing_3}.\\
\begin{proof}{Lemma}{wearegoing_1}   The method is to
approximate $u(.,t_k)$ by an ${\mathcal R}$-valued step function
$v_k(.)$,
to associate to the profile $v_k(.)$ a sequence of configurations
$(\xi^{N,k})_N$ (in the sense \eqref{profile_xi} below), to use
Riemann hydrodynamics for these approximations from time $t_k$ up to
time $t_{k+1}$, and to show that approximated systems at time
$t_{k+1}$ are close enough to the original ones, both at
microscopic and macroscopic levels.\\ \\
Let $v_k(.)$ be an approximation of $u(.,t_k)$ given by Lemma
\ref{cor_approx}, so that
\be \label{uniform_approx}
\Delta(u(.,t_k),v_k(.))\leq\delta\eps,\quad k=0,\ldots,{\mathcal
K}-1 \ee
We write $v_k(.)$ as
\be \label{decomp_approx} v_k=\sum_{l=0}^{l_k}r_{k,l}{\bf
1}_{[x_{k,l},x_{k,l+1})} \ee
where
$ -\infty=x_{k,0}<x_{k,1}<\ldots<x_{k,l_k}<x_{k,l_k+1}=+\infty$,
$r_{k,l}\in{\mathcal R}$, $r_{k,0}=r_{k,l_k}=0$, and, for $1<l\le
l_k$,
\be \label{where} x_{k,l}-x_{k,l-1}\geq\eps \ee
For $t_k\leq t< t_{k+1}$, we denote by $v_k(.,t)$ the entropy
solution to \eqref{hydrodynamics} at time $t$ with Cauchy datum
$v_k(.)$. For $l=1,\ldots,l_k$, define on $\widetilde{\Omega}$ the
%
%
configurations $\xi^{N,k,l}(\omega_0,\omega)$ and
$\xi^{N,k}(\omega_0,\omega)$ by
\be \label{def_xinkl} \xi^{N,k,l}(\omega_0,\omega)(x):=\left\{
\ba{lll} \eta_{Nt_k}(\eta^{r_{k,l-1}}(\omega_0),\omega)(x) &
\mbox{if} & x<[Nx_{k,l}]\\
\eta_{Nt_k}(\eta^{r_{k,l}}(\omega_0),\omega)(x) & \mbox{if} &
x\geq[Nx_{k,l}]
\ea
\right. \ee
\be \label{def_xink}
\xi^{N,k}(\omega_0,\omega)(x):=\eta_{Nt_k}(\eta^{r_{k,l}}(\omega_0),
\omega)(x),\, \mbox{ if }\, [Nx_{k,l}]\leq x<[Nx_{k,l+1}] \ee
so that  $\xi^{N,k}(\omega_0,\omega)$ has finitely many particles,
and
\be \label{notice_xi} \xi^{N,k}(x)=\xi^{N,k,l}(x),\,\mbox{ if }\,
[Nx_{k,l-1}]\leq x< [Nx_{k,l+1}] \ee
Moreover by the commutation property \eqref{mapping_shift},
\beq\label{tau_x-xi}
&&\tau_{[Nx_{k,l}]}\xi^{N,k,l}(\omega_0,\omega)\\&=&
T(\eta_{Nt_k}(\tau_{[Nx_{k,l}]}\eta^{r_{k,l-1}}(\omega_0),
\theta_{[Nx_{k,l}],0}\omega),\eta_{Nt_k}(\tau_{[Nx_{k,l}]}\eta^{r_{k,l}}(\omega_0),
\theta_{[Nx_{k,l}],0}\omega))\nonumber \eeq
Evolutions from \eqref{def_xinkl}--\eqref{def_xink} are denoted by
\beq
\xi^{N,k}_t(\omega_0,\omega)&=&\eta_{t}(\xi^{N,k}(\omega_0,\omega),
\theta_{0,Nt_k}\omega)\nonumber
\\
\label{def_xinklt}
\xi^{N,k,l}_t(\omega_0,\omega)&=&\eta_{t}(\xi^{N,k,l}(\omega_0,\omega),
\theta_{0,Nt_k}\omega) \eeq
so that by the Markov property \eqref{mapping_markov} and
\eqref{tau_x-xi}, we have
\beq\label{tau_x-tau_t-xi}
&&\tau_{[Nx_{k,l}]}\xi^{N,k,l}_t(\omega_0,\omega)\\&=&
T(\eta_t(\tau_{[Nx_{k,l}]}\eta^{r_{k,l-1}}(\omega_0),
\theta_{[Nx_{k,l}],Nt_k}\omega),\eta_t(\tau_{[Nx_{k,l}]}\eta^{r_{k,l}}(\omega_0),
\theta_{[Nx_{k,l}],Nt_k}\omega))\nonumber \eeq
We claim that
\be \label{shift_eq}
\lim_{N\to\infty}\alpha^N(\eta_{Nt_k}(\eta^{r_{k,l}}(\omega_0),\omega))(dx)
=r_{k,l}dx,\quad\widetilde{\Prob}\mbox{-a.s.} \ee
For $k=0$, this follows from Proposition \ref{Rezakhanlou}, since
$\eta^{r_{k,l}}(\omega_0)\sim\nu^{r_{k,l}}$. For
$k=1,\ldots,{\mathcal K}-1$, this follows from Corollary
\ref{corollary_2_2} with $\lambda=\rho=r_{k,l}$ and $s_0=x_0=0$.
Indeed, on the one hand we have $R_{r_{k,l},r_{k,l}}(.,.)\equiv
r_{k,l}$; on the other hand, if
$\omega'=(\eta,\xi,\omega)\sim\bar{\nu}^{r_{k,l},r_{k,l}}\otimes\Prob$,
we have $\eta=\xi$ a.s., so that
$\beta^N_{Nt_k}(\omega')=\alpha^N(\eta_{Nt_k}(\eta,\omega))$ a.s.,
with $(\eta,\omega)\sim\nu^{r_{k,l}}\otimes\Prob$.\\ \\
By \eqref{def_xink}, for every continuous function $\psi$ on $\R$
with compact support,
\begin{eqnarray*}
&&\int_\R\psi(x)\alpha^N(\xi^{N,k}(\omega_0,\omega))(dx)\\ \quad& =
& \sum_{l=1}^{l_k}\int_\R\psi(x){\bf
1}_{[x_{k,l},x_{k,l+1})}(x)\alpha^N(
\eta_{Nt_k}(\eta^{r_{k,l}}(\omega_0),\omega))
(dx) +  O(1/N)
\\
\quad & \stackrel{N\to\infty}{\longrightarrow} &
\int_\R\psi(x)v_k(x)dx,\quad\widetilde{\Prob}\mbox{-a.s.}
\end{eqnarray*}
where the convergence follows from \eqref{shift_eq} and
\eqref{decomp_approx}. Hence,
\be \label{profile_xi}
\lim_{N\to\infty}\alpha^N(\xi^{N,k}(\omega_0,\omega))(dx)
=v_k(.)dx,\quad\widetilde{\Prob}\mbox{-a.s.} \ee
that is, $\xi^{N,k}$ is a microscopic version of $v_k(.)$.
%
%
%
%
For $k=0,\ldots,{\mathcal K}-1$, we write  (remember that
$\eps'=t_{k+1}-t_k$)
\beq
\Delta^N(t_{k+1})-\Delta^N(t_k) & \leq & \Delta\left[ \alpha^N(
\eta^N_{Nt_{k+1}} ),\alpha^N( \xi^{N,k}_{N\eps'} )
\right]-\Delta^N(t_k)\nonumber\\
& + & \Delta\left[ \alpha^N( \xi^{N,k}_{N\eps'}
),v_k(.,\eps') \right]\nonumber\\
& + & \Delta(v_k(.,\eps'),u(.,t_{k+1}))\label{decomp_delta}
\eeq
By \eqref{uniform_approx} and iii), b) of Proposition
\ref{standard},
\be \label{decrease_delta} \Delta(v_k(.,t_{k+1}
-t_k),u(.,t_{k+1}))\leq\Delta(v_k(.),u(.,t_k))\leq\delta\eps \ee
%
 is  a bound for the third term on the r.h.s. of
\eqref{decomp_delta}. For the first term, we define the event

\[
E^{N,k}:=\left\{
\Delta\left[ \alpha^N( \eta^N_{Nt_{k+1}} ),\alpha^N(
\xi^{N,k}_{N\eps'} ) \right]\leq
\Delta\left[ \alpha^N( \eta^N_{Nt_{k}} ),\alpha^N( \xi^{N,k} )
\right]+\delta\eps
\right\} \]
By  assumption  \eqref{finite_2} and Proposition \ref{lemma_bm},
%
\[
\widetilde{\Prob}(\widetilde{\Omega}-E^{N,k})\leq
C'Ne^{-cN\delta\eps} \]
for some constant $C'$ independent of $k$. Thus, by Borel-Cantelli's
lemma, there exists a random $N_1(\omega_0,\omega)$ such that
$\widetilde{\Prob}\mbox{-a.s.}$, $E^{N,k}$ holds for all $N\geq N_1$
and $k=0,\ldots,{\mathcal K}-1$.
On the other hand,
$$
\Delta\left[ \alpha^N( \eta^N_{Nt_{k}} ),\alpha^N( \xi^{N,k} )
\right] \leq \Delta^N(t_k)+
\Delta(u(.,t_k),v_k(.))+\Delta\left[v_k(.),\alpha^N(\xi^{N,k})\right]
$$
Thus, by \eqref{uniform_approx}, \eqref{profile_xi} and (i) of Lemma
\ref{delta_conv},
$$
\limsup_{N\to\infty}\left\{\Delta\left[ \alpha^N( \eta^N_{Nt_{k}}
),\alpha^N( \xi^{N,k} ) \right]-\Delta^N(t_k)\right\}\leq\delta\eps
$$
with $\widetilde{\Prob}$-probability $1$.
Therefore,   $\widetilde{\Prob}$-a.s., for
$k=0,\ldots,{\mathcal K}-1$,
\be\label{decomp_delta_1} \limsup_{N\to\infty}\left\{\Delta\left[
\alpha^N( \eta^N_{Nt_{k+1}} ),\alpha^N( \xi^{N,k}_{N\eps'} )
\right]-\Delta^N(t_k)\right\}\leq 2\delta\eps \ee
%
%
is a bound for the first term on the r.h.s. of \eqref{decomp_delta}.
We finally bound the second term on the r.h.s. of
\eqref{decomp_delta}.
For  $k=0,\ldots,{\mathcal K}-1$ and  $l=1,\ldots,l_k$, by
 the
respective definitions \eqref{extended_shift},
\eqref{def_empirical_shift} of $\theta'_.$, $\beta^N_.$, and
\eqref{def_xinkl}, \eqref{def_xinklt}, \eqref{tau_x-tau_t-xi} we
have
\begin{eqnarray}
 & & ( \tau_{-[Nx_{k,l}]/N})
 \alpha^N\left(\xi^{N,k,l}_{Nt}(\omega_0,\omega)\right)=
 \alpha^N\left(( \tau_{[Nx_{k,l}]})\xi^{N,k,l}_{Nt}(\omega_0,\omega)\right) \nonumber\\
& = &\beta^N_{Nt}\circ\theta'_{[Nx_{k,l}],Nt_k}
(\eta^{r_{k,l-1}}(\omega_0),\eta^{r_{k,l}}(\omega_0),\omega)\label{xi_riemann}
\end{eqnarray}
for any $t\ge 0$. This and Corollary \ref{corollary_2_2} imply
\be\label{local_riemann}
\lim_{N\to\infty}\alpha^N\left(\xi^{N,k,l}_{Nt}\right)
=R_{r_{k,l-1},r_{k,l}}(.-x_{k,l},t)dx,\quad
\widetilde{\Prob}\mbox{-a.s.} \ee
Let us consider the events
\[  F^{N,k,l}:=\left\{
\xi^{N,k}_{N\eps'}(x)=\xi^{N,k,l}_{N\eps'}(x),\, \forall
x\in\Z\cap\left[ N(x_{k,l-1}+v\eps'), N(x_{k,l+1}-v\eps') \right)
\right\}\] 
By \eqref{notice_xi}, the definition \eqref{time_step} of $\eps'$,
\eqref{where} and Lemma \ref{finite_prop}, we have
\[
\widetilde{\Prob}\left(\widetilde{\Omega}-F^{N,k,l}\right)\leq
e^{-CN\eps'} \]
Thus there exists a random $N_2(\omega_0,\omega)$ such that
 $\widetilde{\Prob}\mbox{-a.s.}$, $F^{N,l,k}$ holds for
every $N\geq N_2$,  $k=0,\ldots,{\mathcal K}-1$ and
$l=1,\ldots,l_k$.  This 
combined with \eqref{local_riemann}  implies  that
$\widetilde{\Prob}\mbox{-a.s.}$, the restriction of
$\alpha^N(\xi^{N,k}_{N\eps'})$ to
$(x_{k,l-1}+v\eps',x_{k,l+1}-v\eps')$  converges as $N\to\infty$ to
the restriction of $R_{r_{k,l-1},r_{k,l}}(.-x_{k,l},\eps')dx$. By
\eqref{time_step} and iv) of Proposition \ref{standard}, this
induces
$$
\lim_{N\to\infty}\alpha^N(\xi^{N,k}_{N\eps'})
=v_k(.,\eps')dx,\quad\widetilde{\Prob}\mbox{-a.s.}$$
which, by Lemma \ref{delta_conv}, implies that the second term on
the r.h.s. of \eqref{decomp_delta} converges
$\widetilde{\Prob}$-a.s. to $0$ as $N\to\infty$. 
  Together
with \eqref{decrease_delta} and \eqref{decomp_delta_1}, this
 yields
Lemma \ref{wearegoing_1}.
\end{proof}
%

\begin{proof}{Lemma}{wearegoing_2}
We label $\eta$-particles increasingly from left to right at each
time $Nt_k$, denoting their positions by $(R^{k,i})_{i\in I}$, where
$I$ is a finite set whose cardinal  $|I|$, of order $O(N)$
by  assumption  \eqref{finite_2}, 
is the number of particles in the system. For simplicity we omit the
dependence of the labelling on $N$ in the notation.  The position of
particle $i$ at time $\theta\in[Nt_k,Nt_{k+1}]$ is denoted by
$R^{k,i}_{\theta}$. Let  for any $s,t\in[t_k,t_{k+1}]$,
\[
\Delta_{s,t} := \Delta(\alpha^N(\eta^N_{Ns}),\alpha^N(\eta^N_{Nt}))
= N^{-1}\abs{\sup_{x\in\Z}\sum_{y\leq
x}\left[\eta^N_{Nt}(y)-\eta^N_{Ns}(y)\right]}
\]
Let $z \in \Z$ be a point at which the supremum above is attained.
We can suppose without loss of generality that
\[
N \Delta_{s,t}=\sum_{y\leq z}\eta^N_{Ns}(y) - \sum_{y\leq
z}\eta^N_{Nt}(y).
\]
Therefore to the left of $z$ at time $Ns$ there are at least $N
\Delta_{s,t}$ more particles than at time $Nt$. Let $I_s$ and $I_t$
be the subsets of   $I$ which label the particles to the left of or
at $z$ at times $Ns$ and $Nt$ respectively. Then we have $|I_s| -
|I_t| \geq N \Delta_{s,t}$ which implies $|I_s \backslash I_t |\geq
N \Delta_{s,t}$. Now if $i \in I_s \backslash I_t$, \beq
R_{Nt}^{k,i} > z \mbox{ since } i \notin I_t \label{Rt}\\
R_s^{k,i} \leq z \mbox{ since } i \in I_s \label{Rs} \eeq By
\eqref{Rt}, since we  have  at most $K$ particles per site, $\max_{i
\in I_s \backslash I_t} R_{Nt}^{k,i} \geq z+K^{-1}N \Delta_{s,t}$.
This implies $\max_{i \in I_s \backslash I_t} (R_{Nt}^{k,i}-
R_{Ns}^{k,i}) \geq K^{-1}N \Delta_{s,t}$ by \eqref{Rs},
thus
\[
K \sup_{i \in I} (R_{Ns}^{k,i}-  R_{Nt}^{k,i})\geq N\Delta_{s,t}
\]
and we conclude that
$$
\Delta(\alpha^N(\eta^N_{Ns}),\alpha^N(\eta^N_{Nt}))  \leq  K
N^{-1}\sup_{i\in
I}\abs{R^{k,i}_{Ns}-R^{k,i}_{Nt}}\label{delta_labels}
$$
Proceeding as in the proof of Lemma \ref{lemma_2_2_2_1} it is
possible to construct processes $Q^{k,i}$ and $S^{k,i}$ on the time
interval $[Nt_k,Nt_{k+1}]$ such that
\[
Q^{k,i}_{Nt}\leq R^{k,i}_{Nt}-R^{k,i}_{Nt_k}\leq S^{k,i}_{Nt}
\]
for $t\in[t_k,t_{k+1}]$,  with: $S^{k,i}$ (resp. $Q^{k,i}$) is a
Markov process on $\Z$ starting from $0$ at time $Nt_k$, that jumps
from $x$ to $x+z$ at rate $p(z)\norm{b}_\infty$ only for $z>0$
(resp. only for $z<0$). Therefore,
\beq\nonumber
&&\Prob\left(
\sup_k\sup_{t\in[t_k,t_{k+1}]}
\Delta\left(\alpha^N(\eta^N_{Nt}),\alpha^N(\eta^N_{Nt_k})\right)\geq
C\eps
\right)\\\nonumber
\leq &\dsp\sum_k\sum_{i\in I}&\Prob\left(
\sup_{t\in[t_k,t_{k+1}]}\abs{R^{k,i}_{Nt}-R^{k,i}_{Nt_k}}\geq CN\eps
\right)\\
%
%
%
%
\label{lastprob1} \leq  &\dsp\sum_k\sum_{i\in I}  & \Prob\left(
-Q^{k,i}_{N(t_{k+1}-t_k)}\geq CN\eps
\right)\\
+ & \dsp\sum_k\sum_{i\in I} & \Prob\left(
S^{k,i}_{N(t_{k+1}-t_k)}\geq CN\eps
\right)\label{lastprob2} \eeq
 Since $p(.)$ has finite first moment, by large deviation bounds
for random walks, the constant $C$ can be chosen large enough such
that  the probabilities in \eqref{lastprob1}--\eqref{lastprob2} are
 smaller than $e^{-C'\eps
N}$ for some constant $C'$ (recall that $t_{k+1}-t_k =\eps'$ is a multiple of $\eps$). By Borel Cantelli's lemma  
we conclude that
$$
\limsup_{N\to\infty} \sup_{k=0,\ldots,{\mathcal
K}}\sup_{t\in[t_k,t_{k+1}]}
\Delta\left(\alpha^N(\eta^N_{Nt}),\alpha^N(\eta^N_{Nt_k})\right)\leq
C\eps
$$
for $\eps=\eps_n$ on a set of probability one, which can be chosen
common to all (the countably many) values of $n\in\N$. On this set
we thus have Lemma \ref{wearegoing_2}.
\end{proof}
\subsubsection{General case}\label{General case}
We will relax assumptions  \eqref{finite_0}--\eqref{finite_2} in two steps.\\ \\
 {\em Step one: compact support only.}  We prove Theorem
\ref{th:hydro} when the additional assumptions
\eqref{finite_1}--\eqref{finite_2} are maintained, but
\eqref{finite_0}--\eqref{finite_00} are relaxed.
%
%
Let $T>0$. By approximating the initial profile by
${\mathcal R}$-valued ones, 
we define  a sequence $(u_0^n)_{n\in\N}$ of $[0,K]$-valued functions
satisfying  \eqref{finite_1}--\eqref{finite_2}  and
\eqref{finite_0}--\eqref{finite_00} for fixed $n$, such that
\be \label{delta_approximation}
\lim_{n\to\infty}\Delta(u_0^n,u_0)=0
\ee
and a family of (deterministic) particle configurations
$(\eta^{n,N}_0)_{n\in\N,N\in\N}$ satisfying \eqref{finite_2} for
fixed $n$, such that
\be\label{approx_profiles}
\lim_{N\to\infty}\Delta\left(\alpha^N(\eta^{n,N}_{0}),u_0^n\right)=0
\ee
for each $n\in\N$. 
Indeed,  let us partition $[-R,R]$ into finitely many intervals
$I_{n,k}$ of length at most $\delta_n\to 0$, and set
$$
u_0^n=\sum_k K \mathbf 1_{(x_{n,k},x_{n,k}+\rho_{n,k} l_{n,k}/K)}
$$
where $l_{n,k}$ denotes the length of $I_{n,k}$, $x_{n,k}$ its left
extremity, and $\rho_{n,k}$ the mean value of $u_0$ on $I_{n,k}$.
Then $u_0^n$ has the same mean value as $u_0$ on $I_{n,k}$, hence
$\Delta(u_0^n,u_0)\leq K\delta_n$. Then we define a sequence of
particle configurations associated to $u_0^n$ by
$$
\eta^{n,N}_0(x)=u_0^n\left(\frac{x}{N}\right),\quad\forall x\in\Z
$$
We denote by $u^n(x,t)$ the entropy solution to
\eqref{hydrodynamics} at time $t$ starting from Cauchy datum
$u_0^n$, and by $\eta^{n,N}_t:=\eta_t(\eta^{n,N}_0,\omega)$ the
evolved particle configuration starting from $\eta^{n,N}_0$. By
triangle inequality for $\Delta$,
\beq
\Delta\left(\alpha^N(\eta^{N}_{Nt}),u(.,t)\right) & \leq &
\Delta\left(\alpha^N(\eta^{N}_{Nt}),\alpha^N(\eta^{n,N}_{Nt})\right)\nonumber\\
& + &
\Delta\left(\alpha^N(\eta^{n,N}_{Nt}),u^n(.,t)\right)\nonumber\\
& + & \Delta(u^n(.,t),u(.,t))
\label{triangle_delta} \eeq
%
 We have by iii), b) of Proposition \ref{standard},
\be\label{further_standard}\Delta(u^n(.,t),u(.,t))\leq\Delta(u^n_0,u_0)\ee
 Further, by Subsection \ref{proof_simplified},
\eqref{approx_profiles} implies the analogue of \eqref{wearegoing}
for $\eta^{n,N}_.$, that is
\be\label{restricted_hydro}
\lim_{N\to\infty}\sup_{t\in[0,T]}\Delta\left(\alpha^N(\eta^{n,N}_{Nt}),u^n(.,t)\right)=0\qquad\widetilde{\Prob}\mbox{-a.s.}
\ee
On the other hand, 
\be\label{small_increase}
\Delta\left(\alpha^N(\eta^{N}_{Nt}),\alpha^N(\eta^{n,N}_{Nt})\right)
=\Delta\left(\alpha^N(\eta^{N}_{0}),\alpha^N(\eta^{n,N}_{0})\right)+
\Gamma^{N,n}_{Nt}
\ee
where, by 
Proposition \ref{lemma_bm},  $\Gamma^{N,n}_{Nt}$ is a random
variable which satisfies
$$
\widetilde{\Prob}\left(
\sup_{t\geq 0}\Gamma^{n,N}_{Nt}\geq\gamma
\right)\leq C'Ne^{-cN\gamma},\quad\forall\gamma>0
$$
for some constant $C'>0$ independent of $n$.
Applying Borel-Cantelli's lemma to a vanishing sequence of values of
$\gamma$,
\be\label{bcincrease}\lim_{N\to\infty}\sup_{t\geq
0}\Gamma^{n,N}_{Nt}=0\ee
 $\widetilde{\Prob}\mbox{-a.s.}$. Furthermore,
\begin{eqnarray}
\Delta\left(\alpha^N(\eta^{N}_{0}),\alpha^N(\eta^{n,N}_{0})\right)&\leq&
\Delta\left(\alpha^N(\eta^{N}_{0}),u_0\right)
+\Delta\left(u_0,u^n_0\right)\nonumber\\
&&+\Delta\left(u^n_0,\alpha^N(\eta^{n,N}_{0})\right)\label{furthermore}
\end{eqnarray}
By \eqref{initial_profile_vague},  (i) of Lemma
\ref{delta_conv} and \eqref{approx_profiles}--\eqref{furthermore},

\[
\limsup_{N\to\infty}\sup_{t\in[0,T]}\Delta\left(\alpha^N(\eta^{N}_{Nt}),u(.,t)\right)\leq
2\Delta(u^n_0,u_0)\]
on a subset of $\widetilde{\Omega}$ with
$\widetilde{\Prob}$-probability one, which  can be chosen to be the
same for all (countably many) values of $n\in\N$ and $T>0$. The
conclusion of Theorem \ref{th:hydro} then follows from
\eqref{delta_approximation} and (ii) of Lemma
\ref{delta_conv}.\\ \\
{\em Step two: general case.}  We now finally relax assumptions
\eqref{finite_1}--\eqref{finite_2}, thanks to the finite propagation
property (both at microscopic and macroscopic levels). Consider
$u_0$ and $\eta^N_0$ as in  the statement of  Theorem
\ref{th:hydro}, without any restriction. Let $w=\max(V,v)$, where
$V$ and $v$ are the constants  given  respectively in Proposition
\ref{standard} and Lemma \ref{finite_prop}.  For $n\in\N$,
 we set
$$
u_0^n:=u_0 \mathbf
1_{[-n,n]},\quad\eta^{n,N}_0(x)=\eta^N_0(x)\mathbf 1_{\Z\cap
[-Nn,Nn]}(x)
$$
By Lemma \ref{finite_prop} and Borel-Cantelli's lemma,
$\widetilde{\Prob}\mbox{-a.s.}$ for large enough $N$,
$$
\eta^N_{Nt}(x)=\eta^{n,N}_{Nt}(x),\,\forall t\leq n/(2w),\,\forall
x\in[-Nn/2,Nn/2]\cap\Z
$$
By the previous step, for each $n\in\N$, $\widetilde{\Prob}$-a.s.,
$\alpha^N(\eta^{n,N}_{Nt})$ converges to $u^{n}(.,t)dx$ as
$N\to\infty$,
 uniformly on bounded times intervals.  By iii), (a) of
Proposition \ref{standard}, for every $t\leq n/(2w)$,
$u^{n}(.,t)=u(.,t)$ on $[-Nn/2,Nn/2]$. Thus,  for every
continuous functions $\psi:\R\to\R$ supported on $[-Nn/2,Nn/2]$,
there is an event of $\widetilde{\Prob}$-probability one on which
$$
\int_\R\psi(x)\alpha^N(\eta^N_{Nt})(dx)\to\int_\R\psi(x)u(x,t)dx
$$
uniformly on the time interval $[0,n/(2w)]$. This event can be
chosen to be the same for all values of $n$ and for a countable set
of continuous functions with compact support that is convergence
determining for the vague topology. This establishes the result.
\begin{appendix}
\section{Proof of Corollary \ref{strong-weak}}\label{appendix_strong-weak}
Let $\mu^N_t$ denote the distribution at time $t$ of a Markov
process with generator \eqref{generator}. Assume
$\alpha^N(\eta)(dx)$ converges in $\mu^N_0$-probability to
$u_0(.)dx$, that is, for all $\eps>0$ and every continuous function
$\psi$ on $\R$ with compact support,
\[
\lim_{N\to\infty}
\mu^N_0\left(\left\{\eta:\,\abs{\int_\R\psi(x)\alpha^N(\eta)(dx)-\int
\psi(x)u_0(x)dx}>\eps\right\}\right)=0 \]
Then for every $t>0$, $\alpha^N(\eta)(dx)$ converges in
$\mu^N_{Nt}$-probability to $u(.,t)dx$. This weak law follows from
the {strong}  law in Theorem \ref{th:hydro}. Indeed, by Skorokhod's
representation theorem, we can find a probability space
$(\Omega_0,{\mathcal F}_0,\Prob_0)$ and a sequence $(\eta^N_0)_N$ of
$\bf X$-valued random variables on $\Omega_0$
 such that $\eta^N_0$ has distribution $\mu^N_0$, and
$\alpha^N(\eta^N_0)(dx)$ converges $\Prob_0$-a.s. to $u_0(.)dx$.
\section{Remarks on subadditivity}
\label{remarks_sub}
As outlined below, it would be possible to establish
\eqref{lim_current} in the particular case $\beta=\alpha=0$ by using
the subadditive ergodic theorem as in \cite[Proposition 3]{afs}.
However we cannot use this approach
when $(\beta,\alpha)\neq (0,0)$.\\ \\
%
%
%
%
Let us introduce
\beq \label{def_X_afs}
X_{0,n}(\omega')&:=&{\phi}^v_{n/v}(\omega')-{\varphi}^v_{n/v}(\eta,\omega)
\\ X_{m,n}(\omega')&:=& X_{0,n-m}(\theta'_{m,m/v}\omega')
\nonumber \eeq
then $X_{m,n}$ is the same as defined in equation \cite[(27)]{afs}
and, by \cite[p. 226]{afs}, it satisfies the {\em superadditivity}
property
\be \label{sub_afs} X_{0,n}\geq X_{0,m}+X_{m,n} \ee
(superadditivity is obtained here rather than subadditivity in
\cite{afs}, because we have $\lambda<\rho$ instead of
$\lambda>\rho$). We point out that the proof of \eqref{sub_afs} in
\cite{afs} uses only attractiveness and the fact that we start with
$\eta\leq\xi$, but not the choice of the distribution of
$(\eta,\xi)$. It can thus be generalized from the asymmetric
exclusion process to our setting.
%
%
%
%
Let us now assume that the probability measure on $\Omega'$ is
$\bar{\nu}^{\lambda,\rho}\otimes\Prob$. We can proceed as in
\cite{afs}. Indeed, because $\bar{\nu}^{\lambda,\rho}$ is invariant
for the coupled process, \eqref{extended_shift} implies that
$\bar{\nu}^{\lambda,\rho}\otimes\Prob$ is invariant by the shift
$\theta'_{x,t}$.
 By \eqref{ordered_coupling}, \eqref{sub_afs} is true
$\bar{\nu}^{\lambda,\rho}\otimes\Prob$-a.s. This and Poisson bounds
on the expectation of $X_{0,n}$ imply, by Kingman's subadditive
ergodic theorem, that $n^{-1}X_{0,n}(\omega')$ converges
$\bar{\nu}^{\lambda,\rho}\otimes\Prob$-a.s. On the other hand,
$n^{-1}{\varphi}^v_{n/v}(\eta,\omega)$ converges
$\bar{\nu}^{\lambda,\rho}\otimes\Prob$-a.s. by \eqref{lemma_2_2_2}
 below. Hence,
\be \label{kingman} \bar{\nu}^{\lambda,\rho}\otimes\Prob\mbox{
a.s.},\,\exists \lim_{n\to\infty}n^{-1}{\phi}^v_{n/v}(\omega')\ee
The limit in \eqref{kingman} can then be identified using the
hydrodynamic limit of \cite{bgrs2}, in the same way as \cite{av} is
used in \cite{afs}. We thus obtain a particular case of
\eqref{lim_current} for $\beta=\alpha=0$.
However, the case $(\beta,\alpha)\neq (0,0)$ would require
\be \label{kingman_shift} \bar{\nu}^{\lambda,\rho}\otimes\Prob\mbox{
a.s.},\,\exists
\lim_{n\to\infty}n^{-1}{\phi}^v_{n/v}(\theta'_{[\beta n],\alpha
n}\omega') \ee
for every $\beta\in\R$ and $\alpha\neq 0$. The a.s. limit
\eqref{kingman} only implies a limit in probability for the shifted
current in \eqref{kingman_shift}, as the distribution of a single
current is unchanged by the shift. In contrast the {\em joint}
distribution of the {\em sequence} of currents may change from
\eqref{kingman} to \eqref{kingman_shift}: Thus we cannot simply
derive \eqref{kingman_shift} from \eqref{kingman}. On the other
hand, the shifted currents
$Y_{0,n}:=X_{0,n}\circ\theta'_{[\beta n],\alpha n}$
no longer enjoy a super-additivity property like \eqref{sub_afs}, so
we cannot use the subadditive ergodic theorem to obtain
\eqref{kingman_shift}. Our approach to obtain \eqref{kingman_shift}
overcomes this difficulty by avoiding the use of subadditivity.
\section{Proof of Proposition \ref{prop_ergo}}\label{appendix_ergo}
The main ingredient is a two-dimensional
extension of Birkhoff's ergodic theorem:
\begin{proposition}\label{prop_ravi}
Let $({\bf\mathcal X},{\mathcal F}, {P})$ be a probability space and
$T,S:{\bf\mathcal X}\to{\bf\mathcal X}$ two measurable mappings such
that  $P\circ T^{-1}=P\circ S^{-1}=P$. Then, for every bounded
$\mathcal F$-measurable $f:{\bf\mathcal X}\to{\bf\mathcal X}$, the
limit
\be\label{ergolim}
f_{**}({x}):=\lim_{n\to\infty} \frac{1}{n} \sum_{j=1}^n \frac{1}{n}
\sum_{i=1}^n f(S^i T^j x)
\ee
exists for almost every ${x}\in\mathcal X$ with respect to $P$.
\end{proposition}
This follows from more general results established for instance in \cite{wie} or \cite[Chapter 6]{kr}.
However we include a  simpler  proof adapted to our case.\\
\begin{proof}{proposition}{prop_ravi}
By ergodic theorem there exists bounded functions $f_*$ and $f_{**}$
such that
\begin{eqnarray}\label{ergothm_1} \frac{1}{n} \sum_{i=1}^n f( S^i x)  & :=
 & f_*^n({x}) \to f_*({x}), ~ P-a.s. \\
\label{ergothm_2} \frac{1}{n}\sum_{j=1}^n f_*(T^j x) & := &
f^n_{**}({x})\to f_{**}(x), ~ P-a.s. \end{eqnarray}
By \eqref{ergothm_1} and Egorov's theorem, there is a sequence of
subsets $A_k\in{\mathcal F}$, $k\in\N$,  such that $\nu(A_k)\to 0$
as $k\to\infty$, and $f_*^n\to f_*$ uniformly on ${\mathcal
X}\backslash A_k$.
Let
$$
F^n_{**}(x):= \frac{1}{n} \sum_{j=1}^n \frac{1}{n} \sum_{i=1}^n
f({S^i T^jx})=\frac{1}{n}\sum_{j=1}^n f^n_*(T^j x)$$
Then,
\begin{eqnarray*}
\abs{F^n_{**}(x)-f_{**}(x)} & \leq & \frac{1}{n}\abs{
\sum_{j=1}^n\left[
f^n_*(T^j x)-f_*(T^j x)\right]1_{{\mathcal X}\backslash A_k}(T^j x)
}\\
& + & \abs{
\frac{1}{n}\sum_{j=1}^n\left[
f^n_*(T^j x)-f_*(T^j x)\right]1_{A_k}(T^j x)
}\\
& + & \abs{
\frac{1}{n}\sum_{j=1}^n f_*(T^j x)-f_{**}(x)
}\\
& \leq & \sup_{{\mathcal X}\backslash A_k}\abs{f^n_*-f_*}+2M
g^n_{A_k}(x)+\abs{f^n_{**}(x)-f_{**}(x)}=:B_{n,k}(x)
\end{eqnarray*}
where $M:=\sup_{\mathcal X}\abs{f}$ and, by ergodic theorem,
\be\label{conv_g}g^n_{A_k}(x):=\frac{1}{n}\sum_{j=1}^n 1_{A_k}(T^j
x)\stackrel{n\to\infty}{\to} g_{A_k}(x),\quad P\mbox{-a.s.}\ee
for some bounded, nonnegative, $\mathcal F$-measurable $g_{A_k}$
such that
\be\label{integral_g}\int g_{A_k}
d\nu=\nu(A_k)\stackrel{k\to\infty}{\to}0\ee
By uniform convergence of $f^n_*$ on  ${\mathcal X}\backslash A_k$,
 a.s. convergence \eqref{ergothm_2}, and  \eqref{conv_g},
$\limsup_{n\to\infty}B_{n,k}(x)\leq 2M g_{A_k}(x)$ holds $P$-a.s. By
\eqref{integral_g}, $g_{A_k}$ goes to $0$ in $L^1(P)$ as
$k\to\infty$. Thus it has a subsequence converging to $0$ $P$-a.s.
Letting $k\to\infty$ along this subsequence concludes the proof.
\end{proof}
\mbox{}\\
\begin{proof}{proposition}{prop_ergo}
%
\\ \\
{\em Existence of the limit.}
Define the random variables
$X_{i,j}:=\int_{(j-1)a}^{ja}f(\tau^{i-1}\eta_s)ds$, where $i,j\in\N$
and $(\eta_s)_{s\geq 0}$ is the stationary Markov process with
generator $L$ and initial distribution $\mu$. Take ${\mathcal
X}=\R^{\N\times\N}$, $\mathcal F$ the product Borel $\sigma$-field,
  $P$ the distribution of the ${\mathcal X}$-valued random
variable $(X_{i,j})_{i,j\in\N}$,
$T\left[(x_{i,j})_{i,j\in\N}\right]=(x_{i,j+1})_{i,j\in\N}$,
$S\left[(x_{i,j})_{i,j\in\N}\right]=(x_{i+1,j})_{i,j\in\N}$. We have
$P\circ T^{-1}=P$ because  $(\eta_s)_{s\geq 0}$  is stationary, and
$P\circ S^{-1}=S$  because $\mu$ and $L$ are invariant by $\tau$.
Then the
existence of the limit follows from Proposition \ref{prop_ravi}.\\
\\
{\em Identification of the limit.}
Let now ${\mathcal X}$ be the Skorokhod space of $\bf X$-valued
paths,  and $P=P_\mu$ the law of the Markov process with generator
$L$ and initial distribution $\mu$.  We consider on $\mathcal X$ the
space shifts $(\tau_x)_{x\in\Z}$ and time shifts $(T_t)_{t\geq 0}$
defined as follows: if $\eta_.=(\eta_s)_{s\geq 0}\in{\mathcal X}$,
then $\tau_x\eta_.:=(\tau_x\eta_s)_{s\geq 0}$, where $\tau_x$ on the
r.h.s. is the spatial shift on particle configurations defined in
Section \ref{sec_results}, and $T_t\eta_.:=(\eta_{t+s})_{s\geq 0}$.
 What follows is a generalization of a standard result for
one-parameter Markov processes (see e.g. \cite[Chapter 7]{b}). By
the above existence step, we can define
\be\label{ergolim_2}
f_{**}(\eta_.):=\lim_{n\to\infty}F_{**}^n(\eta_.) \ee
$P_\mu$-a.s., where
\[F_{**}^n(\eta_.)=\frac{1}{an}
\int_0^{an} \frac{1}{n} \sum_{i=1}^n \tau^i f(\eta_t)dt\]
As a limit of measurable functions, $f_{**}$ is measurable. For
every $t>0$, $T_t F_{**}^n-F_{**}^n$ and $\tau F_{**}^n-F_{**}^n$
 consist of space-time sums over boundary domains of order
$O(n)=o(n^2)$, hence in the limit $n\to\infty$, $f_{**}$ is
invariant by  $(T_t)_{t\geq 0}$ and $\tau:=\tau_1$.  To show that
this implies $f_{**}$ is a $P$-a.s. constant function, we will prove
that any measurable subset $F$ of $\mathcal X$ which is invariant by
$(T_t)_{t\geq 0}$  and $\tau$ has $P_\mu$-probability $0$ or $1$.
Taking expectations in \eqref{ergolim_2}, the constant value of
$f_{**}$ must be $\int f \,d\mu$, and Proposition \ref{prop_ergo} is
thus established.\\ \\
%
%
Let $F\subset\mathcal X$ be measurable, and invariant by
$(T_t)_{t\geq 0}$  and $\tau$. Set
\be\label{gergo} g(\eta)=P_\mu(F|\eta_0=\eta)=:P_\eta(F) \ee
which is defined for $\mu$-a.e. $\eta$. Here, $P_\eta$ denotes the
law of the Markov process starting from deterministic state
$\eta\in{\bf X}$. We are going to prove that
\be\label{statement_g}g\equiv 0\mbox{ or }g\equiv 1\ee
$\mu$-a.s., which will imply $P(F)=\int_{\bf
X}g(\eta)\mu(d\eta)\in\{0,1\}$.\\ \\
We have
\[g(\tau\eta)=P_{\tau\eta}(F)=P_{\eta}(\tau^{-1}F)=P_{\eta}(F)=g(\eta)\]
where the second equality follows from translation invariance
\eqref{translation_gen} of $L$ (which implies $P_{\tau\eta}=\tau
P_\eta$), and the third from $\tau$-invariance of $F$. Therefore $g$
is $\mu$-a.s. invariant by the spatial shift $\tau$. We claim that
$g=\mathbf 1_G$ $\mu$-a.s. for some $G\subset{\bf X}$. Indeed, let
 ${\mathcal F}_t$ denote the $\sigma$-field of  ${\mathcal X}$
  generated by the mappings $\eta_.\mapsto\eta_s$ for $s\leq t$.
With $P_\mu$-probability one,
$$ g(\eta_t) = P_{\eta_t}(F)= P_\mu(T_t^{-1}F|{\mathcal F}_t) =
P_\mu(F|{\mathcal F}_t)$$
where the second equality follows from Markov property, and the
third from the $T_t$ invariance of $F$. By the martingale
convergence theorem, we have the $P_\mu$-a.s. limit
\be \label{convergence_g}
\lim_{t\to\infty} g(\eta_t)=\mathbf 1_F(\eta_.)
\ee
Since $\eta_t\sim\mu$ for all $t\ge 0$, for every $\eps>0$,
\be\label{convergence_gg} P_\mu(\eta_.:\,\eps\le g(\eta_t)\le
1-\eps)= \mu(\eta:\,\eps\le g(\eta)\le 1-\eps)\ee
we conclude from \eqref{convergence_g}--\eqref{convergence_gg} that
the law of $g(\eta_t)$ is Bernoulli, hence $g=\mathbf 1_G$
$\mu$-a.s. for some $G\subset{\bf X}$. The desired conclusion
\eqref{statement_g} is thus equivalent to $\mu(G)\in\{0,1\}$, which
we now establish.\\ \\
First we claim that, with $P_\mu$-probability one, we have
$g(\eta_t)=g(\eta_0)$ for all $t>0$. Indeed, by $T_t$ invariance of
$F$, Markov property and definition of $G$,
\beq\nonumber P_\mu\left(
\{\eta_t\not\in G\}\cap F\right) & = & P\left(\{\eta_t\not\in
G\}\cap T_t^{-1}F\right)\\\nonumber
& = & \int_{\{\eta_t\not\in G\}\subset{\mathcal
X}}P_{\eta_t}(F)P_\mu(d\eta_.)\\\nonumber
& = & \int_{\{\eta_t\not\in G\}\subset{\mathcal
X}}g(\eta_t)P_\mu(d\eta_.)\\\label{nullprob_1}
& = & 0
\eeq
%
%
Similarly we have
\be\label{nullprob_2} P_\mu\left(
\{\eta_t\in G\}\cap({\mathcal X}\backslash F)\right)=0 \ee
 It follows from \eqref{nullprob_1}--\eqref{nullprob_2} that
$\{\eta_t\in G\}=F$ up to a set of $P_\mu$-probability $0$. In other
words, for every $t>0$, we have $g(\eta_t)=g(\eta_0)=\mathbf
1_F(\eta_.)$ with $P_\mu$-probability one.  \\ \\
Now, assume $\mu(G)\not\in\{0,1\}$, and define the conditioned
measures $\mu_G(d\eta):=\mu(d\eta|\eta\in G)$ and $\mu_{{\mathcal
X}\backslash G}(d\eta):= \mu(d\eta|\eta\in{\mathcal X}\backslash
G)$, so that  $\mu=\mu(G)\mu_G+\mu({\mathcal X}\backslash
G)\mu_{{\mathcal X}\backslash G}$. The invariance of $L$ and $\mu$
with respect to space shift imply the same for $\mu_G$ and
 $\mu_{{\mathcal X}\backslash G}$. A simple computation,
using invariance of $\mu$ for $L$, and the fact that $g(\eta_t)$ is
$P_\mu$ a.s. constant, shows that $\mu_G$ and  $\mu_{{\mathcal
X}\backslash G}$  are also invariant for $L$. This contradicts the
fact that $\mu\in({\mathcal I}\cap{\mathcal S})_e$.
\end{proof}
\end{appendix}
\mbox{}\\ \\
\noindent {\bf Acknowledgments:} We thank Tom Mountford for
 helpful discussions.  We
acknowledge the support of the French Ministry of Education through
the ANR BLAN07-2184264 grant.  K.R. was supported by NSF grant DMS
0104278.  K.R. and E.S. thank TIMC - TIMB, INP Grenoble. K.R. thanks
Universit\'{e} de Rouen for hospitality.  H.G. thanks IXXI for
partial support. Part of this work was done during the semester
``Interacting Particle Systems, Statistical Mechanics and
Probability Theory'' at IHP, where K.R. benefited from a CNRS
``Poste Rouge''.
\end{document}